\documentclass{amsart}
\usepackage[normalem]{ulem}
\usepackage{
  amssymb,
  epsfig,
  mathrsfs,
  mathpazo,
  stackengine,
  scalerel,
  url,
  bbm,
  esint
}
\usepackage{thmtools}

\makeatletter
\@ifundefined{newcounteralias}{}{%
  \renewcommand\thmt@autorefsetup{%
    \@xa\def
      \csname\thmt@envname autorefname\@xa\endcsname
      \@xa{\thmt@thmname}%
  }%
}
\makeatother
\usepackage{xcolor}
\usepackage{stmaryrd}
\usepackage[multiple]{footmisc}
\usepackage{accents} 
\usepackage{mathtools} 
\usepackage{derivative}

\DeclareFontFamily{U}{mathx}{}
\DeclareFontShape{U}{mathx}{m}{n}{<-> mathx10}{}
\DeclareSymbolFont{mathx}{U}{mathx}{m}{n}
\DeclareMathAccent{\widehat}{0}{mathx}{"70}
\DeclareMathAccent{\widecheck}{0}{mathx}{"71}

\makeatletter
\newcommand*\bigcdot{\mathpalette\bigcdot@{.5}}
\newcommand*\bigcdot@[2]{\mathbin{\vcenter{\hbox{\scalebox{#2}{$\m@th#1\bullet$}}}}}
\DeclareRobustCommand{\Cpp}{C\nolinebreak[4]\hspace{-.05em}\raisebox{.15ex}{\texttt{++}}}
\makeatother

\newcommand{\osc}{{\rm osc}}

\usepackage{arydshln, booktabs, makecell, pbox, tabularx}

\usepackage{graphicx, caption, subcaption}

\newtheorem{theorem}{Theorem}[section]
\newtheorem{lemma}[theorem]{Lemma}
\newtheorem{proposition}[theorem]{Proposition}

\declaretheorem[style=definition,qed=$\vartriangle$,sibling=theorem]{example}
\declaretheorem[style=remark,qed=$\vartriangle$,sibling=theorem]{remark}
\numberwithin{equation}{section}

\newcommand{\R}{\mathbb R}

\newcommand{\N}{\mathbb N}

\newcommand{\cP}{\mathcal P}
\newcommand{\cE}{\mathcal E}

\newcommand{\cL}{\mathcal L}

\newcommand{\Lis}{\cL\mathrm{is}}

\newcommand{\identity}{\mathrm{Id}}

\DeclareMathOperator{\ran}{ran}
\DeclareMathOperator{\supp}{supp}

\DeclareMathOperator{\diam}{diam}
\DeclareMathOperator*{\argmin}{argmin}
\DeclareMathOperator{\dist}{dist}

\newcommand*\diff{\mathop{}\!\mathrm{d}}

\DeclareMathOperator{\divv}{div}

\renewcommand{\mod}{\,\mathrm{mod}\,}

\newcommand{\private}[1]{}

\newcommand{\be}{\begin{equation}}
\newcommand{\ee}{\end{equation}}

\newcommand{\bbP}{\mathbb{P}}
\newcommand{\tria}{{\mathcal T}}
\newcommand{\pria}{{\mathcal P}}
\newcommand{\RT}{\mathit{RT}}

\newcommand{\uumlaut}{{\"u}}

\newcommand{\gen}{{\tt gen}}

\DeclareMathOperator*{\fiint}{\ensuremath{\iint\text{\kern-1.36em{\raisebox{5.87pt}{\rotatebox{-93}{$\setminus$}}}}}}

\newenvironment{algotab}%
{\par\begin{samepage}%
\begin{tabbing}\ttfamily%
 \hspace*{5mm}\=\hspace{3ex}\=\hspace{3ex}\=\hspace{3ex}\=\hspace{3ex}%
\=\hspace{3ex}\=\hspace{3ex}\=\hspace{3ex}\=\hspace{3ex}\kill}%
{\end{tabbing}\end{samepage}}

\newcounter{ccondition}

{%
\renewcommand{\theequation}{\temp}%
\addtocounter{equation}{-1}%
\par\noindent%
\ignorespacesafterend%
}

\newcounter{mylistcounter}
\renewcommand{\themylistcounter}{(\roman{mylistcounter})}

\makeatletter
\@namedef{subjclassname@2020}{%
  \textup{2020} Mathematics Subject Classification}
\makeatother

\title{A double-adaptivity solver for parabolic PDEs}
\date{\today}

\author{Gregor Gantner}
\address{Department of Applied Mathematics, University of Twente, P.O. Box 217, 7500 AE Enschede, The Netherlands}
\email{gregor.gantner@utwente.nl}
\author{Robin Smeets}
\address{Korteweg-de Vries (KdV) Institute for Mathematics, University of Amsterdam, P.O. Box 94248, 1090 GE Amsterdam, The Netherlands.}
\email{r.k.h.smeets@uva.nl}
\author{Rob Stevenson}
\address{Korteweg-de Vries (KdV) Institute for Mathematics, University of Amsterdam, P.O. Box 94248, 1090 GE Amsterdam, The Netherlands.}
\email{rob.p.stevenson@gmail.com}

\thanks{GG acknowledges funding by the Deutsche Forschungsgemeinschaft (DFG, German Research Foundation) under Germany’s Excellence Strategy - EXC-2047/1 - 390685813 and by the Nederlandse Organisatie voor Wetenschappelijk Onderzoek (NWO, Dutch Research Council) through the Vidi project Optimal adaptive space-time boundary and finite element methods (OASTMethods) with file number VI.Vidi.243.148 under the grant https://doi.org/10.61686/LTBER75172.}



\begin{document}

\begin{abstract} We study minimal residual space-time finite element discretizations of linear parabolic initial value problems in canonical space-time variational form.
To deal with the arising dual norm, we introduce the Riesz lift of the residual as an additional variable.
Quasi-optimality of the primal variable of the mixed system follows from a uniform inf-sup condition.
This condition is known to be satisfied for finite element spaces w.r.t.~prismatic partitions of the space-time cylinder that allow for a decomposition into time-slabs.
We prove that this condition cannot be expected to hold otherwise.
To recover stability for general partitions and the data at hand, we derive an a posteriori condition on the error between the exact Riesz lift of the residual and its Galerkin approximation --- being the secondary variable of our system --- under which 
the primal variable is quasi-optimal.
We derive a posteriori error estimators for both variables, and use them in a double-adaptive loop that alternates test-space with trial-space enrichment. 
We illustrate our findings with numerical experiments in \(1+1\) and \(2+1\) dimensions.
\end{abstract}
\maketitle

\section{Introduction}
This paper is about the adaptive numerical solution of linear parabolic evolution equations in a simultaneous space-time variational formulation.
The monolithic approach based on such a variational formulation has shown to be advantageous for parallel implementations
\cite{169.06,306.65,169.065}, it guarantees quasi-best approximations from the employed trial spaces, and it gives the possibility of adaptive local refinements 
 simultaneously in time and space. At the downside, monolithic methods require more memory than time-stepping methods. This disadvantage vanishes for applications where the forward parabolic problem is coupled with an adjoint backward parabolic problem as with
problems of optimal control \cite{169.053, 19.96, 75.292,75.2573}, or with goal-oriented error minimization \cite{169.052}. Another interesting application of space-time methods is in reduced basis methods for parameter-dependent problems where, other than with time-stepping methods, they reduce complexity both in space and time \cite{75.528, 75.292,138.24}.

In this work, we consider the parabolic initial value problem as a boundedly invertible operator w.r.t.~the canonical domain $X=L_2(\Xi;V) \cap H^1(\Xi;V')$
and codomain $Y' \times H$, where $\Xi=(0,T)$, $Y=L_2(\Xi;V)$, and $V \hookrightarrow H \hookrightarrow V'$ is the Gelfand triple associated to the spatial partial differential operator.
Following \cite{11,249.99,249.992}, given finite-dimensional trial and test subspaces $X^\delta \subset X$ and $Y^\delta \subset Y$, we minimize the residual over $X^\delta$, 
where the norm on $Y'$ is replaced by the discretized dual norm $\sup_{0 \neq v \in Y^\delta}\frac{\bigcdot(v)}{\|v\|_Y}$. By introducing the Riesz lift of the corresponding residual component as an additional variable, one arrives at an equivalent saddle-point system. 

The resulting numerical approximation is quasi-best, uniformly in  all data, if and only if the saddle-point system is uniformly inf-sup stable. For  $X^\delta$ and $Y^\delta$ that are tensor products of suitable temporal and spatial finite element spaces, uniform inf-sup stability was shown in \cite{11}. In \cite{249.99}, this result has been generalized to the time-slab setting in which on different slabs possibly different spatial finite element spaces are employed.
Also this last setting, however, does not allow for general local refinements of the space-time mesh, being an important motivation for the use of monolithic space-time methods instead of classical time-marching methods.

Unfortunately, as will be demonstrated in this work, for finite element partitions that do not permit a partition into time-slabs, uniform inf-sup stability \emph{cannot} be expected. 
In view of this result, following an idea from \cite{45.44}, we derive an \emph{a posteriori condition} on the size of the test space for quasi-optimality for the \emph{data at hand}. 
Furthermore, we derive an a posteriori estimator for the error in the principal variable in the $X$-norm that is efficient and reliable when the test space satisfies the aforementioned condition for quasi-optimality.
The combination of both these results suggests a \emph{double-adaptivity} loop. In the inner loop, the test space is expanded until it satisfies the condition for quasi-optimality, whereas in the outer loop the trial space is expanded guided by the a posteriori estimator.

The condition on the test space for quasi-optimality involves the difference in the $Y$-norm between the exact Riesz lift of the residual of the primal variable and its Galerkin approximation from $Y^\delta$. Using the technique of equilibrated fluxes \cite{70.8}, we develop an estimator for this error that, modulo higher order oscillation terms, is efficient and reliable.
The construction of this estimator requires an auxiliary refined partition, whose cardinality cannot be bounded by a constant multiple of the cardinality of the original partition. This has the undesirable consequence that the computation of the estimator can in general not be performed in linear complexity. We therefore also investigate a simpler, non-certified surrogate in which the exact Riesz lift is approximated in the one-degree enriched space w.r.t.~the original partition. The resulting difference in the $Y$-norm is not a guaranteed upper bound, but the numerical experiments indicate that it remains of the same order as the certified estimator.

The numerical experiments confirm the role of the two adaptive mechanisms. For singular solutions caused by non-matching initial and boundary data, or by reduced regularity of the initial datum, the adaptive algorithm refines towards the relevant parts of the boundary. In $1+1$ dimensions this leads, for the non-matching initial datum, to the rate $0.5$ for all tested polynomial degrees, whereas uniform refinement gives only a rate of about $0.12$. For the reduced boundary-regularity datum, the adaptive rates increase with the polynomial degree and reach approximately $0.88$ and $0.93$ for $p=2$ and $p=3$, respectively. In $2+1$ dimensions the gains are more modest but remain systematic. The experiments also show that the higher-order surrogate for the condition on the test space tracks the certified estimator within the same order of magnitude.

There are other pairs of spaces w.r.t.~which the parabolic initial value problem is well-posed.
A competitor to our approach is the first-order system formulation 
introduced in \cite{75.257}, further studied in \cite{75.28,75.291}, where the codomain is an $L_2$-type space so that it allows for a straightforward residual minimization over any finite-dimensional subspace. Compared to the norm on $L_2(0,T;V) \cap H^{1}(0,T;V')$, however, the norm on the domain of the first-order operator is stronger, which, for non-tensor product trial spaces (cf.~\cite{75.291,75.2592}), can affect the rates that can be achieved for singular solutions. For example, for the non-matching initial datum in $1+1$ dimensions, our adaptive $p=1$ computation attains the rate $0.5$, i.e.~the rate expected for smooth solutions in the $X$-norm, whereas \cite{75.291} reports a rate $0.43$.

A further advantage of our approach is that it extends to parabolic problems with a non-linear spatial operator that is Lipschitz continuous and strongly monotone. This will be the topic of a forthcoming work that builds on the combination of \cite{19.98} and the current work. 

\subsection{Organization}
In Sect.~\ref{sec:ivp}, the parabolic initial value problem is written in a well-posed variational form. The minimal residual discretization is presented in Sect.~\ref{sec:minres}, where it is shown that its solution is quasi-best under a uniform inf-sup condition that requires that the test space is sufficiently large in relation to the trial space.
In Sect.~\ref{sec:inf-sup}, it is recalled that for suitable finite element trial and test spaces whose underlying partitions permit a subdivision into time-slabs, the uniform inf-sup condition is valid. In Appendix~\ref{sec:necessity}, it is shown that, without the time-slab condition, uniform inf-sup stability cannot be expected.
In Sect.~\ref{sec:data-dependent}, an a posteriori data-dependent condition for quasi-optimality is derived.
It involves the error between the exact Riesz lift of the residual and its Galerkin approximation from the test space. An a posteriori estimator for this error is derived in Sect.~\ref{sec:apost}. A double-adaptivity loop in which alternately the trial and test spaces are expanded is presented in Sect.~\ref{sec:double-adaptive}.
In Sect.~\ref{sec:numerics}, numerical experiments are presented, and conclusions can be found in Sect.~\ref{sec:conclusion}.

\subsection{Notations}  \label{sec:notations}
In this work, by $C \lesssim D$ we will mean that $C$ can be bounded by a multiple of $D$, independently of parameters that $C$ and $D$ may depend on.
Obviously, $C \gtrsim D$ is defined as $D \lesssim C$, and $C\eqsim D$ as $C\lesssim D$ and $C \gtrsim D$. 

We set $\N:=\{1,2,\ldots\}$ and $\N_0:=\{0\} \cup \N$.

For normed linear spaces $E$ and $F$, by $\cL(E,F)$ we will denote the normed linear space of bounded linear mappings $E \rightarrow F$,
and by $\Lis(E,F)$ its subset of boundedly invertible linear mappings $E \rightarrow F$.
We write $E \hookrightarrow F$ to denote that $E$ is continuously embedded in $F$.
For simplicity only, we exclusively consider linear spaces over the scalar field $\R$.
For a Hilbert space $K$, $R_K\in \Lis(K,K')$ will denote the Riesz isometry defined by $(R_K v)(w)=\langle w,v\rangle_K$.

\section{Parabolic initial value problem} \label{sec:ivp}
For a Gelfand triple $V \hookrightarrow H \simeq H' \hookrightarrow V'$, and a $\mathrm{T}>0$, with $\Xi:=(0,\mathrm{T})$ we consider the \emph{Initial Value Problem} (IVP) of finding $u\colon \Xi \rightarrow V$ such that
$$
\left\{
\begin{array}{rcll}
\tfrac{\diff u}{\diff t} (t)+A(t)u(t) & = & \ell(t) & \text{a.e. }t \in \Xi  , \\
u(0) & = & u_0,
\end{array}
\right.
$$
where $\ell \in L_2(\Xi;V')$, $A(t) \in \cL(V,V')$ bounded uniformly in a.e.~$t \in \Xi$, and $u_0 \in H$.

We consider this problem in \emph{operator}, or \emph{space-time variational form}: We set
$$
Y:=L_2(\Xi;V), \quad X:= Y \cap H^1(\Xi;V'), 
$$
and $A\in \cL(Y,Y')$, $\diff_t \in \cL(X,Y')$ by
$$
 (Aw)(v):=\int_\Xi (A(t)w(t))(v(t))\diff t,\quad (\diff_t w)(v):=\int_\Xi (\tfrac{\diff w}{\diff t} (t))(v(t)) \diff t.
 $$
It is known (e.g., \cite[Ch.~1, Thm.~3.1]{185}) that
\be \label{eq:embedding}
X \hookrightarrow C(\overline{\Xi};H),
\ee
so that the trace operator $\gamma_t\colon X \rightarrow H\colon u \mapsto u(t)$ is bounded uniformly in $t \in \overline{\Xi}$.
For $w, z \in L_2(\Xi;V) \cap H^1(\Xi;H)$,
$$
(\diff_t w)(z)+(\diff_t z)(w)=\int_\Xi \tfrac{\diff}{\diff t} \langle w(t),z(t)\rangle_H \diff t=\langle w(\mathrm{T}),z(\mathrm{T})\rangle_H-\langle w(0),z(0)\rangle_H,
$$
so that by continuous extension using the density of $L_2(\Xi;V) \cap H^1(\Xi;H)$ in $X$,
$$
\diff_t+\diff_t'+\gamma_0' \gamma_0=\gamma_{\mathrm{T}}' \gamma_{\mathrm{T}}
$$
as mappings $X \rightarrow X'$.

Given data $(\ell,u_0) \in Y' \times H$, our IVP now reads as finding $u \in X$ such that
$$
\left\{
\begin{array}{rcll}
(Bu)(v):=(\diff_t u)(v)+(Au)(v)&=&\ell(v) & (v \in Y),\\
\gamma_0 u &=& u_0,
\end{array}
\right.
$$
i.e., as
\be \label{eq:33}
B_e u:=\left[\begin{array}{@{}c@{}} B \\ \gamma_0 \end{array}\right] u=\left[\begin{array}{@{}c@{}} \ell \\ u_0\end{array}\right]=:f.
\ee
In order to ensure existence and uniqueness of the solution of this problem, we assume that $A(t) \in \cL(V,V')$ is
 \emph{uniformly coercive} i.e.,
$$
(A(t) v)(v) \gtrsim \|v\|_{V}^2\quad(v \in V,\, \text{a.e. } t \in \Xi).
$$
Under this condition\footnote{Actually, it already suffices that $A(t)$ satisfies a uniform G{\aa}rding inequality.} the following result is well-known (e.g.,~\cite[Ch.~3, Thm.~4.1]{185},
\cite[Ch.~IV, \S26]{314.9}).
\begin{theorem} It holds that  $B_e \in \Lis(X,Y' \times H)$.
\end{theorem}

We set
$$
A_s:=\tfrac12 (A+A'),\quad A_a:=\tfrac12 (A-A'),
$$
and equip $Y$ with
\begin{align*}
\|\bigcdot\|_Y&:=\sqrt{(A_s \bigcdot)(\bigcdot)},
\intertext{so that $A_s\in \Lis(Y,Y')$ is an isometry, and $X$ with}
\|\bigcdot\|_X&:=\sqrt{\|\bigcdot\|_Y^2+\|\diff_t\bigcdot\|_{Y'}^2+\|\gamma_{\mathrm{T}}\bigcdot\|_H^2},
\end{align*}
both being norms that are equivalent to the canonical norms on $Y$ and $X$, respectively.

\section{Minimal residual discretization} \label{sec:minres}
Let $(X^\delta)_{\delta \in \Delta}$ and $(Y^\delta)_{\delta \in \Delta}$ be families of closed non-trivial subspaces of $X$ and $Y$, respectively.
Sometimes we will refer to $X^\delta$ and $Y^\delta$ as the \emph{trial} and \emph{test} spaces.
We include $\infty \in \Delta$ for which $X^{\infty}=X$ and $Y^{\infty}=Y$, which has the following consequence:\vspace*{.5ex}
\begin{center}
\emph{All results that will be shown for $\delta \in \Delta$ include the continuous setting.}
\end{center}
We will assume that $X^\delta$ and $Y^\delta$ for $\delta \neq\infty$  are finite-dimensional. The trivial embeddings $X^\delta \hookrightarrow X$ and $Y^\delta \hookrightarrow Y$ will be denoted by $E_X^\delta$ and $E_Y^\delta$, respectively.
For clarity, we will often write these operators and their duals explicitly.

We equip $Y^\delta$ with its natural (Hilbertian) norm
\begin{align*}
&\|\bigcdot\|_{Y^\delta}:=\|E_Y^\delta \bigcdot\|_Y,
\intertext{so that}
&\|\bigcdot\|_{(Y^\delta)'}=\sup_{0 \neq v \in Y^\delta}\frac{\bigcdot(v)}{\|v\|_{Y^\delta}}.
\intertext{On $X$, we define \emph{`mesh'-dependent} norms}
&\|\bigcdot\|_{X,\delta}:=\sqrt{\| \bigcdot\|^2_{Y}+\|(E_Y^\delta)' \diff_t \bigcdot\|^2_{(Y^\delta)'}+\|\gamma_{\mathrm{T}} \bigcdot\|_H^2}\,,
\intertext{and equip $X^\delta$ with}
&\|\bigcdot\|_{X^\delta}:=\|E_X^\delta \bigcdot\|_{X,\delta}\,,
\intertext{so that}
&\|\bigcdot\|_{(X^\delta)'}=\sup_{0 \neq v \in X^\delta}\frac{\bigcdot(v)}{\|v\|_{X^\delta}}.
\end{align*}
We emphasize that for $\delta\neq\infty$, the norm on $X^\delta$, and therefore that on $(X^\delta)'$, depend on the choice of the subspace $Y^\delta \subset Y$, which will be clear from the context.

For $\delta=\infty$, the norms $\|\bigcdot\|_{Y^{\infty}}$, $\|\bigcdot\|_{(Y^{\infty})'}$, $\|\bigcdot\|_{X^{\infty}}$, $\|\bigcdot\|_{(X^{\infty})'}$ will usually be written without the index $\infty$. 
Thanks to \eqref{eq:embedding}, $\|\bigcdot\|_X=\sqrt{\|\bigcdot\|^2_{Y}+\|\diff_t \bigcdot\|^2_{Y'}+\|\gamma_{\mathrm{T}} \bigcdot\|_H^2}$ is equivalent to the canonical norm $\sqrt{\|\bigcdot\|^2_{Y}+\|\diff_t \bigcdot\|^2_{Y'}}$ on $X$.

The following result was proven in \cite[Prop.~2.2]{249.992} for the case that $\delta=\infty$, but its proof extends directly to any $\delta \in \Delta$.
It generalizes the well-known `inf-sup identity' from, e.g.,~\cite{70.95} by allowing for both $A_a \neq 0$ and $\delta \neq\infty$.

\begin{theorem}  \label{thm:3}
With $\alpha:=\|A_a\|_{\cL(Y,Y')}=r_\sigma(A_s^{-1} A_a)$, 
$c(\alpha):=1+\frac{\alpha}{2}\big(\alpha+\sqrt{\alpha^2+4}\,\big)$, 
 for $0 \neq w \in X \cap Y^\delta$ it holds that
$$
\frac{\|{E_Y^\delta}' Bw\|_{{Y^\delta}'}^2+\|\gamma_0 w\|_H^2}{\|w\|_{X,\delta}^2} \in \Big[\tfrac{1}{c(\alpha)},c(\alpha)\Big].
$$
\end{theorem}
\medskip

Our \emph{minimal residual approximation} $u^\delta \in X^\delta$ for the solution $u \in X$ of \eqref{eq:33} is defined as
\be \label{ls}
u^\delta:=\argmin_{w \in X^\delta} \|{E_Y^\delta}'(\ell-B E_X^\delta w)\|_{{Y^\delta}'}^2+ \|u_0-\gamma_0 E_X^\delta w\|_H^2.
\ee
The corresponding Euler--Lagrange equation reads as
\be \nonumber
\begin{split}
({E_X^\delta}' B' E_Y^\delta ({E_Y^\delta}'A_s E_Y^\delta)^{-1} {E_Y^\delta}' B E_X^\delta&+{E_X^\delta}'  \gamma_0' \gamma_0 E_X^\delta) u^\delta\\
&=
{E_X^\delta}' B' E_Y^\delta ({E_Y^\delta}'A_s E_Y^\delta)^{-1} {E_Y^\delta}' \ell+{E_X^\delta}'  \gamma_0' u_0,
\end{split}
\ee
or equivalently, by setting $\lambda^\delta=\lambda^\delta(u^\delta):=({E_Y^\delta}'A_s E_Y^\delta)^{-1}{E^{\delta}_Y}'( \ell-B E^\delta_X u^{\delta}) \in Y^\delta$, as
\be \label{m9}
\left[\begin{array}{@{}ccc@{}}{E_Y^\delta}'A_s E_Y^\delta & {E_Y^{\delta}}' B E^\delta_X\\ {E^\delta_X}' B' E_Y^{\delta}& -{E^\delta_X}'  \gamma_0' \gamma_0 E^\delta_X \end{array}\right]
\left[\begin{array}{@{}c@{}} \lambda^{\delta} \\ u^{\delta} \end{array}\right]=
\left[\begin{array}{@{}c@{}} {E^{\delta}_Y}' \ell \\ -{E^\delta_X}'  \gamma_0' u_0 \end{array}\right].
\ee
The theory of saddle-point problems shows that \eqref{m9}, and so \eqref{ls}, has a unique solution if and only if
$$
(\gamma_\delta^{B_e})^2:=\inf_{0 \neq w \in X^\delta} 
\frac{\|{E_Y^\delta}' B E^\delta_X w\|^2_{{Y^\delta}'}+\|\gamma_0 E_X^\delta w\|_H^2}{\|B E^\delta_X w\|^2_{Y'}+\|\gamma_0 E_X^\delta w\|_H^2}>0.
$$

\noindent In that case, \cite[Thm.~3.3]{204.19}, together with $\|B_e\|_{\cL(X,Y'\times H)}\|B_e^{-1}\|_{\cL(Y'\times H,X)} \leq c(\alpha)$ by Theorem~\ref{thm:3},
shows that
$$
\|u-E_X^\delta u^\delta\|_X \leq \frac{c(\alpha)}{\gamma_\delta^{B_e}} \inf_{w \in X^\delta} \|u-E_X^\delta w\|_X.
$$

We will \emph{always} assume that for all $\delta \in \Delta$,
$$
X^\delta \subset Y^\delta.
$$
Consequently, by using the equivalence of $\|\bigcdot\|_X$ and $\|\bigcdot\|_{X,\delta}$ on the finite-dimensional space $X^\delta$, 
another consequence of Theorem~\ref{thm:3} is that $\gamma_\delta^{B_e}>0$.
In particular, \eqref{m9}, and so \eqref{ls}, has a unique solution.
\medskip

Ideally, the choice of the families $(X^\delta)_{\delta \in \Delta}$ and $(Y^\delta)_{\delta \in \Delta}$ should guarantee the \emph{uniform stability}
\be \label{eq:infsup1}
\inf_{\delta \in \Delta} \gamma_\delta^{B_e}>0.
\ee
The constant $\gamma_\delta^{B_e}$ is the inf-sup constant associated to the operator $B_e\colon X \rightarrow Y' \times H$ and trial and test spaces $X^\delta$ and $Y^\delta \times H$. 
As such, it can be characterized in terms of approximability of the `optimal test space' $\ran R_{Y' \times H}^{-1} B_e E_X^\delta$ by $Y^\delta \times H$ \cite[Prop.~2.5]{35.8565}, or alternatively, by the norm of a Fortin interpolator related to $B_e$ and $X^\delta$ and $Y^\delta \times H$ \cite[Prop.~5.1]{249.992}. Due to the non-trivial structure of $B_e$, both characterizations, however, do not give easy access to the constant $\gamma_\delta^{B_e}$.
Fortunately, Theorem~\ref{thm:3} allows us to bound $\gamma_\delta^{B_e}$ from below by a multiple of the inf-sup constant $ \gamma_\delta^{\diff_t}$ associated to $\diff_t\colon X \rightarrow Y'$ and $X^\delta$ and $Y^\delta$. Indeed, setting 
\be \label{eq:56}
\gamma^{\diff_t}_\delta:=\inf_{\{w \in X^\delta\colon \diff_t E_X^\delta w \neq 0\}} \frac{\|{E_Y^\delta}'\diff_t E^\delta_X w\|_{{Y^\delta}'}}{\|\diff_t E^\delta_X w\|_{Y'}},
\ee
and noticing that $\|\bigcdot\|_{X^\delta} \geq \gamma_\delta^{\diff_t}\|E_X^\delta \bigcdot\|_X$, thanks to $X^\delta \subset Y^\delta$ an application of Theorem~\ref{thm:3}  shows that 
\be \label{eq:lower}
\gamma_\delta^{B_e} \geq \frac{\gamma^{\diff_t}_\delta}{c(\alpha)}.
\ee
As we will see in Sect.~\ref{sec:inf-sup}, $\gamma^{\diff_t}_\delta$ can be shown to be uniformly positive in tensor product, and more generally, time-slab settings. For these settings, we conclude the following result.

\begin{theorem}[Quasi-optimality] \label{thm:quasi} If
\be \label{eq:infsup2}
\gamma^{\diff_t}_\Delta:=\inf_{\delta \in \Delta} \gamma^{\diff_t}_\delta >0,
\ee
then\footnote{Using the equivalence of the minimal residual discretization and the so-called BEN method, in \cite[Thm.~3.1]{249.992} an upper bound was found that is sharper for $A_a \neq 0$.}
$$
\|u-E_X^\delta u^\delta\|_X \leq \frac{c(\alpha)^2}{\gamma^{\diff_t}_\Delta} \inf_{w \in X^\delta} \|u-E_X^\delta  w\|_X.
$$
\end{theorem}

Another application of \eqref{eq:infsup2} is in \emph{a posteriori error estimation}.  
It is known, e.g.~\cite[Prop.~5.1]{249.992}, that $\gamma_\delta^{B_e}>0$ is equivalent to the existence of a Fortin projector $\Pi^\delta\colon Y\times H \rightarrow Y^\delta\times H$ characterized by $(B_e X^\delta)(\ran (\identity-\Pi^\delta))=0$ and $\|\Pi^\delta\|_{\cL(Y\times H,Y\times H)} =\frac{1}{\gamma_\delta^{B_e}}$.
Using \eqref{eq:lower}, we can estimate $\frac{1}{\gamma_\delta^{B_e}} \leq \frac{c(\alpha)}{\gamma^{\diff_t}_\Delta}$. The following theorem builds on \cite[Thm.~2.1]{35.93556}, but it provides better constants in the upper bound.

\begin{theorem} \label{thm:4}
Assume \eqref{eq:infsup2}. Then, the solution $(\lambda^\delta,u^\delta)$ of \eqref{m9} satisfies
$$
\tfrac{1}{\sqrt{c(\alpha)}} \eta(\delta)\leq \|u-E_X^\delta u^\delta\|_X \leq \tfrac{c(\alpha) \sqrt{c(\alpha)}}{\gamma_\Delta^{\diff_t}} \eta(\delta)+\sqrt{c(\alpha)}\|(\identity-{\Pi^\delta}' )f\|_{Y'\times H},
$$
where $\eta(\delta):=\sqrt{\|\lambda^\delta\|_{Y^\delta}^2+\|u_0-\gamma_0 E_X^\delta u^\delta\|_H^2}$.
\end{theorem}

\begin{proof} The result follows from a combination of the following statements.

For $w \neq u \in X$, Theorem~\ref{thm:3} gives that $\frac{\|f-B_e w\|^2_{Y' \times H}}{\|u-w\|^2_X} \in [\frac{1}{c(\alpha)},c(\alpha)]$.

For $w \in X^\delta$ and $v \in Y \times H$, it holds that
$$
|(f-B_e E_X^\delta w)(v)| \leq |(f-B_e E_X^\delta w)(\Pi^\delta v)|+|f((\identity-\Pi^\delta)v)|,
$$
and thus, using $\|\Pi^\delta\|_{\cL(Y\times H,Y\times H)} \le \frac{c(\alpha)}{\gamma_\Delta^{\diff_t}}$,
\begin{align*}
\sup_{0 \neq v \in Y^\delta \times H} \hspace*{-0.5em} \frac{|(f-B_e E_X^\delta w)(v)|}{\|v\|_{Y \times H}}
& \leq \|f-B_e E_X^\delta w\|_{Y' \times H}=\sup_{0 \neq v \in Y \times H} \hspace*{-0.5em} \frac{|(f-B_e E_X^\delta w)(v)|}{\|v\|_{Y \times H}}
\\
&\leq \tfrac{c(\alpha)}{\gamma^{\diff_t}_\delta} \sup_{0 \neq v \in Y^\delta \times H} \hspace*{-1em} \frac{|(f-B_e E_X^\delta w)(v)|}{\|v\|_{Y \times H}}+
\|(\identity-{\Pi^\delta}') f\|_{Y' \times H}.
\end{align*}

It holds that
\begin{align*}
\sup_{0 \neq v \in Y^\delta \times H} \hspace*{-0.5em}\frac{|(f-B_e E_X^\delta w)(v)|}{\|v\|_{Y \times H}}
&=
 \hspace*{-0.5em} \sup_{0 \neq (v_1,v_2) \in Y^\delta \times H} \hspace*{-1em}\frac{|(\ell-B E_X^\delta w)(v_1)+\langle u_0 -\gamma_0 E_X^\delta w,v_2\rangle_H|}{\sqrt{\|v_1\|_{Y}^2+\|v_2\|_H^2}}\\
&=
\sqrt{\|{E_Y^\delta}'(\ell-B E_X^\delta w)\|_{{Y^\delta}'}^2+\|u_0-\gamma_0 E_X^\delta w\|_H^2}.
\end{align*}

Finally, taking $w = u^\delta$, it holds that $\|{E_Y^\delta}'(\ell-B E_X^\delta u^\delta)\|_{{Y^\delta}'}=\|\lambda^\delta\|_{Y^\delta}$.
\end{proof}

The above theorem shows that the a posteriori estimator $\eta(\delta)$ for $\|u-u^\delta\|_X$ is \emph{efficient}. Because of the presence of the `data-oscillation' term $\sqrt{c(\alpha)} \|(\identity-{\Pi^\delta}' )f\|_{Y'\times H}$, it is however not clear whether it is \emph{reliable}. If $\inf_{\delta \in \Delta} \gamma_\delta^{B_e}>0$ 
had been established by \emph{constructing} a Fortin interpolant, then this construction might have been such that this data-oscillation term can be shown to be of higher order than
$\inf_{w \in X^\delta} \|u-w\|_X$. 
Consequently, the estimator would be at least asymptotically reliable (see, e.g., \cite{35.93556,204.195} for examples).

In the current case, we only know that a Fortin projector exists. Since for any $w \in X^\delta$, we have
$(\identity-{\Pi^\delta}' )f=(\identity-{\Pi^\delta}' )(B_e (u-E_X^\delta w))$, 
and, because $Y^\delta \times H$ is a non-trivial, proper subspace of $Y \times H$, it holds that $\|\identity-\Pi^\delta\|_{\cL(Y\times H ,Y \times H)}=\|\Pi^\delta\|_{\cL(Y\times H,Y \times H)}$ (\cite{169.5,315.7}),
we can estimate
\begin{align}
\begin{split}
\label{eq:osc}
\sqrt{c(\alpha)}\|(\identity&-{\Pi^\delta}' )f\|_{Y'\times H} \\ & \leq \sqrt{c(\alpha)} \|\identity-\Pi^\delta\|_{\cL(Y \times H,Y \times H)} \|B_e\|_{\cL(X,Y' \times H)} \inf_{w \in X^\delta}\|u-E_X^\delta w\|_X\\
&\leq \sqrt{c(\alpha)} \frac{c(\alpha)}{\gamma_\Delta^{\diff_t}} \sqrt{c(\alpha)}\inf_{w \in X^\delta}\|u-E_X^\delta w\|_X,
\end{split}
\end{align}
which, however, still does not guarantee reliability.

Therefore, let us make the \emph{saturation assumption} that, for some constant $0<q<\big(\frac{c(\alpha)^2}{\gamma_\Delta^{\diff_t}}\big)^{-1}$, for any $\delta \in \Delta$ there exists a $\widetilde{\delta}\in \Delta$ with
$X^\delta \subset X^{\widetilde{\delta}}$, $Y^\delta \subset Y^{\widetilde{\delta}}$ such that
$$ 
\inf_{\widetilde{w} \in X^{\widetilde{\delta}}}\|u-E_X^{\widetilde{\delta}} \widetilde{w}\|_X \leq q \inf_{w \in X^\delta}\|u-E_X^\delta w\|_X.
$$
Then by \emph{replacing} the test space $Y^\delta$ by $Y^{\widetilde{\delta}}$ for the computation of the minimal residual approximation $u^\delta \in X^\delta$, 
\eqref{eq:osc} with $\Pi^\delta$ replaced by $\Pi^{\widetilde{\delta}}$ and $X^\delta$ by $X^{\widetilde{\delta}}$ shows that the data-oscillation term can be bounded by $\frac{c(\alpha)^{2}}{\gamma_\Delta^{\diff_t}} q \inf_{w \in X^\delta}\|u-E_X^\delta w\|_X$, which gives reliability.

It is, however, well-known, that the saturation assumption cannot hold uniformly over all data $f=(\ell,u_0) \in Y' \times H$ \cite{68.5}.
Later, in Proposition~\ref{prop:1}, we will provide a (data-dependent) a posteriori verifiable condition for the estimator $\eta(\delta)$ to be efficient \emph{and} reliable.

\section{Inf-sup stability condition \eqref{eq:infsup2}} \label{sec:inf-sup}
We briefly summarize results from \cite{11,249.99,249.992,249.991} about situations for which validity of  the crucial stability condition \eqref{eq:infsup2} has been verified.
We consider the case that for some bounded Lipschitz domain $\Omega \subset \R^d$, 
$$
(H,V)=(L_2(\Omega),H^1_0(\Omega)).
$$
In this section we equip $Y=L_2(\Xi;V)$ with its canonical tensor product norm instead of $\sqrt{(A_s \bigcdot)(\bigcdot)}$. Because both norms are equivalent, the resulting inf-sup constants $\gamma_\delta^{\diff_t}$ are uniformly equivalent, and so is $\gamma_\Delta^{\diff_t}=\inf_{\delta \in \Delta}\gamma_\delta^{\diff_t}$.

\begin{example}[tensor-product setting] \label{ex:tensor}
For $\delta \in \Delta$, let 
$$
Y^\delta = Y_t^\delta \otimes Y_x^\delta,\quad X^\delta = X_t^\delta \otimes X_x^\delta.
$$
Then, a tensor-product argument shows that with
$$
\gamma_{\Delta,t}:=\inf_{\delta \in \Delta} \! \inf_{\{z_t \in X_t^\delta\colon \tfrac{\diff z_t}{\diff t} \neq 0\}}\hspace*{-1.2em}\frac{\sup_{0 \neq v_t \in Y_t^\delta} \frac{\int_\Xi \tfrac{\diff z_t}{\diff t} v_t\diff t}{{\|v_t\|_{L_2(\Xi)}}}}{\|\tfrac{\diff z_t}{\diff t}\|_{L_2(\Xi)}}, \,\,
\gamma_{\Delta,x}:=\inf_{\delta \in \Delta}\inf_{0 \neq z_x \in X_x^\delta}\hspace*{-0.8em}\frac{\sup_{0 \neq v_x \in Y_x^\delta} \frac{\langle z_x, v_x\rangle_H}{\|v_x\|_V}}{\|z_x\|_{V'}},
$$
it holds that $\gamma^{\diff_t}_\Delta \geq \gamma_{\Delta,t} \gamma_{\Delta,x}$.
Clearly, $\gamma_{\Delta,t} >0$ if $\frac{\diff}{\diff t} X_t^\delta \subset Y_t^\delta$ (which is, however, not a necessary condition, see \cite[Prop.~6.1]{11} for an example).

If $Y_x^\delta=X_x^\delta$, then $\gamma_{\Delta,x}=\inf_{\delta \in \Delta} \|P^\delta\|_{\cL(V,V)}^{-1}$ where $P^\delta$ is the $H$-orthogonal projector onto $X_x^\delta$.
Uniform boundedness in the $H^1(\Omega)$-norm of the $L_2(\Omega)$-orthogonal projector onto finite element spaces $X_x^\delta$ is an intensively studied subject.
Recently, in \cite{64.596} it was shown that for $d<7$ this uniform boundedness holds true for the family of all Lagrange finite element spaces of arbitrary degree w.r.t.~all conforming newest vertex bisection meshes that can be created from an initial mesh.
\end{example}

Above example shows that \eqref{eq:infsup2} can be realized for trial and test spaces being finite element spaces w.r.t.~partitions of the space-time cylinder 
$$
\Sigma:=\Xi \times \Omega
$$
that are products of a partition of $\Xi$ and a partition of $\Omega$. The next example shows that this setting can be generalized to the situation that in different time intervals different spatial partitions are employed.

\begin{example}[time-slab setting] \label{ex:time-slab}
Let $(\bar{Y}^\delta, \bar{X}^\delta)_{\delta \in \bar{\Delta}}$ be a family of pairs of closed subspaces of $Y$ and $X$
for which $\gamma^{\diff_t}_{\bar{\Delta}}=\gamma^{\diff_t}_{\bar{\Delta}} \big((\bar{Y}^\delta \times \bar{X}^\delta)_{\delta \in \bar{\Delta}}\big)>0$, e.g.,~a family as in Example~\ref{ex:tensor}.
Suppose that for $\delta \in \Delta$, $Y^\delta$ and $X^\delta$ are such that for some
finite partition $\Xi^\delta=([t_{i-1}^\delta,t_{i}^\delta])_i$ of $\Xi$, with
$G^\delta_i(t):=t_{i-1}^\delta+\frac{t}{\mathrm{T}}(t_{i}^\delta-t_{i-1}^\delta)$ and arbitrary $\delta_i \in \bar{\Delta}$, it holds that
\begin{align*}
Y^\delta \supseteq \{v \in L_2(\Xi;V) &\colon v|_{(t^\delta_{i-1},t^\delta_i)} \circ G^\delta_i \in \bar{Y}^{\delta_i}\},\\
X^\delta \subseteq \{u \in C(\overline{\Xi};V) &\colon u|_{(t^\delta_{i-1},t^\delta_i)} \circ G^\delta_i \in \bar{X}^{\delta_i}\}.
\end{align*}
Then, it holds that $\gamma^{\diff_t}_\delta \geq \gamma^{\diff_t}_{\bar{\Delta}}>0$ (see \cite{249.99,249.992}).
\end{example}

\begin{remark} Another generalization of the `full' tensor-product setting from Example~\ref{ex:tensor} outside the field of finite elements is that to `sparse' tensor products, see \cite{11}.
With the application of \emph{wavelets-in-time} instead of finite elements, this `sparse' setting has been generalized by allowing (adaptive)  `refinements' (more accurately called `enrichments') that are simultaneously localized in time \emph{and} space, see \cite{249.991}. 
\end{remark}

\section{A data-dependent condition for quasi-optimality}  \label{sec:data-dependent}
As we have seen in Theorem~\ref{thm:quasi}, under the uniform inf-sup condition \eqref{eq:infsup2}, our minimal residual approximation $u^\delta \in X^\delta$, i.e., the solution of \eqref{ls}, is, w.r.t.~the norm on $X$, a quasi-best approximation to the solution $u$ of our parabolic initial value problem. 

Since $u$ may have singularities that are simultaneously localized in time and space, we would like to consider families of finite element partitions that contain partitions that are locally refined simultaneously in time and space. Such partitions, however, do \emph{not} allow a subdivision into time-slabs.
In Appendix~\ref{sec:necessity}, it is demonstrated that, because of the latter, \emph{neither \eqref{eq:infsup2} nor \eqref{eq:infsup1} can be expected to be valid}.
\medskip

There is a one-to-one relation between \eqref{eq:infsup2} and uniform approximability in the $Y$-norm of the optimal test space $\ran A_s^{-1} \diff_t E_X^\delta$ by $Y^\delta$ (cf.~\eqref{eq:59} and use $R_Y^{-1}=A_s^{-1}$). In this section, we show that quasi-optimality of $u^\delta$ holds also true under 
an \emph{a posteriori} condition \eqref{eq:pjotr} on the distance between $A_s^{-1} (\ell-B E_X^\delta u^\delta)$ and $Y^\delta$. 
As \eqref{eq:infsup2}, validity of \eqref{eq:pjotr} will require that $Y^\delta$ is sufficiently large in relation to $X^\delta$. 
While \eqref{eq:infsup2} gives quasi-optimality of $u^\delta \in X^\delta$ \emph{uniformly in all data} $(\ell,u_0) \in Y' \times H$, \eqref{eq:pjotr} guarantees this \emph{only for the data at hand}.
\medskip

The idea of the data-dependent condition has been introduced in \cite{45.44} in the setting of least-squares solvers for convection-diffusion problems. In a somewhat modified form it has also been investigated in \cite{64.18}.
The setting considered in \cite{45.44} is slightly different from ours. In our framework it corresponds to the case where $u_0=0$, so that for
$X^\delta \subset X_0:=\{w \in X\colon \gamma_0 w=0\}$, \eqref{ls} reads as $u^\delta:=\argmin_{w \in X^\delta} \|{E_Y^\delta}'(\ell-B E_X^\delta w)\|_{{Y^\delta}'}^2$.
This difference however is harmless. Relevant though is that in the following Theorem~\ref{thm:pjotr} \emph{any} constant $\varrho\geq 0$ is allowed, whereas the value of the corresponding constant 
in \cite[Lemma 3.3]{45.44} is restricted to $[0,2)$.
  
 \begin{theorem} \label{thm:pjotr} Given $(\ell,u_0) \in Y' \times H$, let $(\lambda^\delta,u^\delta) \in Y^\delta \times X^\delta$ be the solution of \eqref{m9} . With $\lambda_\delta=\lambda_\delta(u^\delta):=A_s^{-1}(\ell-B  u^\delta)$, suppose that for some $\varrho \geq 0$,
 \be \label{eq:pjotr}
 \|\lambda_\delta- \lambda^\delta\|_Y \leq \varrho \sqrt{\| \lambda^\delta\|_Y^2+(1+\tfrac{1}{(\varrho+\frac12)^2}) \|u_0-\gamma_0 u^\delta\|_H^2)}.
 \ee
 Then, it holds that
  \be \label{eq:g12}
\|u- u^\delta\|_X \leq c(\alpha) \sqrt{\varrho^2+\tfrac{9 \varrho+4}{4 \varrho+4}} \inf_{\bar{u}^{\delta} \in X^\delta} \|u- \bar{u}^\delta\|_X.
 \ee
 where $u:=B_e^{-1}(\ell,u_0)$.
 \end{theorem}
 
 \begin{proof} We will demonstrate that
 \be \label{eq:1}
 \begin{split}
 \|\ell - B u^\delta\|_{Y'}^2+&\|u_0-\gamma_0  u^\delta\|_H^2 \\
 &\leq (\varrho^2+\tfrac{9 \varrho+4}{4 \varrho+4}) \inf_{\bar{u}^{\delta} \in X^\delta}
  \|\ell - B  \bar{u}^\delta\|_{Y'}^2+\|u_0-\gamma_0 \bar{u}^\delta\|_H^2,
 \end{split}
 \ee
 so that the proof is completed by applying Theorem~\ref{thm:3}.

 In variational form, \eqref{m9} reads as solving $(\lambda^\delta,u^\delta) \in Y^\delta \times X^\delta$ from
 \begin{alignat}{4}
\label{eq:2}  (A_s \lambda^\delta)(v)+(B u^\delta)(v)&&&=\,&&\ell(v)&&\quad(v \in Y^\delta),\\
\label{eq:3}  (B w)(\lambda^\delta)+\langle u_0-\gamma_0 u^\delta,\gamma_0 w\rangle_H &&&=\,&& 0 &&\quad(w \in X^\delta).
 \intertext{Further, we have}
\label{eq:4} (A_s \lambda_\delta)(v)+(B u^\delta)(v)&&&=\,&&\ell(v)&&\quad(v \in Y).
 \end{alignat}

 The definition of $\|\bigcdot\|_Y$, and \eqref{eq:4} and \eqref{eq:2} show that
 $$
 \|\ell-Bu^\delta\|^2_{Y'}=(\ell-B u^\delta)(\lambda^\delta)+(\ell-B u^\delta)(\lambda_\delta-\lambda^\delta),\quad \|\lambda^\delta\|_Y^2=(\ell-B u^\delta)(\lambda^\delta).
 $$
Consequently, for $K>0$, we have
\begin{align*}
\|\ell&-B u^\delta\|_{Y'}^2+K \|\lambda^\delta\|_Y^2+(K+1)\|u_0-\gamma_0 u^\delta\|_H^2\\
&=(K+1)\big[(\ell-B u^\delta)(\lambda^\delta)+\|u_0-\gamma_0 u^\delta\|_H^2\big]+(\ell-B u^\delta)(\lambda_\delta-\lambda^\delta).
\end{align*}
By our assumption \eqref{eq:pjotr} and Young's inequality, for $\kappa>0$, we have
\begin{align*}
(\ell-B u^\delta)(\lambda_\delta-\lambda^\delta) & \leq \varrho \|\ell-B u^\delta\|_{Y'} \sqrt{\|\lambda^\delta\|^2_Y+(1+\tfrac{1}{(\varrho+\frac12)^2})\|u_0-\gamma_0 u^\delta\|^2_H}\\
&  \leq \varrho \big[ \kappa \|\ell-B u^\delta\|^2_{Y'}+\tfrac{1}{4 \kappa} (\|\lambda^\delta\|^2_Y+(1+\tfrac{1}{(\varrho+\frac12)^2})\|u_0-\gamma_0 u^\delta\|_H^2)\big] 
 \end{align*}
 
 By taking $\kappa=\frac{1}{2\varrho +1}$, and using \eqref{eq:3} with $w=\bar{u}^\delta-u^\delta$ for arbitrary $\bar{u}^\delta \in X^\delta$, we obtain
\begin{align*}
\tfrac{1+\varrho}{1+2\varrho} &\|\ell-B u^\delta\|_{Y'}^2+(K-\tfrac{\varrho(2\varrho+1)}{4})\|\lambda^\delta\|_Y^2+(K+1-\tfrac{\varrho(2\varrho+1)}{4}(1+\tfrac{1}{(\varrho+\frac12)^2}))\|u_0-\gamma_0 u^\delta\|_H^2\\
& \leq (K+1)\big[(\ell-B u^\delta)(\lambda^\delta)+\langle u_0-\gamma_0 u^\delta, u_0-\gamma_0 u^\delta \rangle_H \big]\\
& = (K+1)\big[(\ell-B \bar{u}^\delta)(\lambda^\delta)+\langle u_0-\gamma_0 u^\delta, u_0-\gamma_0 \bar{u}^\delta \rangle_H \big]\\
& \leq (K+1)\big[\|\ell-B \bar{u}^\delta\|_{Y'}\|\lambda^\delta\|_Y+\|u_0-\gamma_0 u^\delta\|_H \|u_0-\gamma_0 \bar{u}^\delta \|_H \big]\\
& \leq \tfrac{K+1}{2} \big[\|\ell-B \bar{u}^\delta\|_{Y'}^2 + \|\lambda^\delta\|_Y^2+\|u_0-\gamma_0 u^\delta\|_H^2+ \|u_0-\gamma_0 \bar{u}^\delta \|^2_H \big],
\end{align*}
i.e.,
\begin{align*}
\tfrac{1+\varrho}{1+2\varrho} &\|\ell-B u^\delta\|_{Y'}^2+\big(\tfrac{K-1}{2}-\tfrac{\varrho(2\varrho+1)}{4}\big)\|\lambda^\delta\|_Y^2+\big(\tfrac{K+1}{2}-\tfrac{\varrho(2\varrho+1)}{4}(1+\tfrac{1}{(\varrho+\frac12)^2})\big)\|u_0-\gamma_0 u^\delta\|_H^2\\
&  \leq \tfrac{K+1}{2} \big[\|\ell-B \bar{u}^\delta\|_{Y'}^2 + \|u_0-\gamma_0 \bar{u}^\delta \|^2_H \big],
\end{align*}
By selecting $K=\frac{\varrho(2\varrho+1)}{2}+1$, the last inequality reads as 
$$
\tfrac{1+\varrho}{1+2\varrho} \big(\|\ell-B u^\delta\|_{Y'}^2+\|u_0-\gamma_0 u^\delta\|_H^2\big) \leq \tfrac{\varrho(2\varrho+1)+4}{4} \big[\|\ell-B \bar{u}^\delta\|_{Y'}^2 + \|u_0-\gamma_0 \bar{u}^\delta \|^2_H \big].
$$
By multiplying this  equation by $\tfrac{1+2\varrho}{1+\varrho}$, we arrive at \eqref{eq:1}.
 \end{proof}
 
 A secondary benefit of imposing condition \eqref{eq:pjotr} is that it guarantees that the a posteriori estimator 
 $\eta(\delta)^2:=\|\lambda^\delta\|_{Y^\delta}^2+\|u_0-\gamma_0 u^\delta\|_H^2$ from Theorem~\ref{thm:4} is efficient \emph{and} reliable.
 
 \begin{proposition} \label{prop:1} Assuming \eqref{eq:pjotr}, it holds that
  $$
 \tfrac{1}{c(\alpha)}\eta(\delta)^2 \leq \|u-u^\delta\|_X^2 \leq c(\alpha)\big(1+\varrho^2(1+\tfrac{1}{(\varrho+\frac12)^2})\big) \eta(\delta)^2.
 $$
 \end{proposition}
 \begin{proof} The proof follows from
 $$
\tfrac{1}{c(\alpha)} \|u-u^\delta\|^2_X \leq  \|\ell-B u^\delta\|_{Y'}^2+\|u_0-\gamma_0  u^\delta\|_H^2 \leq c(\alpha)\|u-u^\delta\|_X^2,
$$
 \begin{align*}
 \|\ell-B u^\delta\|_{Y'}^2+\|u_0-\gamma_0  u^\delta\|_H^2=& \|\lambda_\delta\|_Y ^2+\|u_0-\gamma_0  u^\delta\|_H^2\\
 = &\|\lambda_\delta- \lambda^\delta\|_Y^2+\| \lambda^\delta\|^2_Y+ \|u_0-\gamma_0  u^\delta\|_H^2,
 \end{align*}
 in combination with the upper bound for $\|\lambda_\delta- \lambda^\delta\|_Y^2$ provided by \eqref{eq:pjotr}.
  \end{proof}
 
 Finally in this section, we show that for any (fixed) $X^\delta$ (i.e.,~any finite-dimensional subspace of $X$) condition \eqref{eq:pjotr} is always enforced by taking $Y^\delta \supset X^\delta$ sufficiently large.

\begin{proposition} \label{prop:2} Let $(\delta_i)_{i \in \N}$ be such that $Y^{\delta_i} \subset Y^{\delta_{i+1}}$ and $\overline{\bigcup_{i \in \N} Y^{\delta_i}}=Y$. Then, for $(\ell, u_0) \in Y' \times H$ and a (fixed) $\varrho>0$, condition \eqref{eq:pjotr} is valid for  $Y^\delta  \supset X^\delta \cup Y^{\delta_i}$ when $i$ is sufficiently large.
\end{proposition}

\begin{proof} A comparison of \eqref{eq:2} and \eqref{eq:4} shows that $\lambda^\delta$ is the Galerkin (and thus best) approximation from $Y^\delta$ to the solution $\lambda_\delta \in Y$ of
$$
A_s \lambda_\delta=\ell-B u^\delta.
$$
Although $u^\delta$ depends on $Y^\delta$, it holds that  $\|u^\delta\|_Y$ is bounded uniformly in $Y^\delta \supset X^\delta$.
Indeed, with $B_e^\delta:=({E_Y^\delta}' B,\gamma_0)\colon X^\delta \rightarrow {Y^\delta}' \times H$, an equivalent formulation of \eqref{ls} is
$\langle B_e^\delta u^\delta, B_e^\delta w\rangle_{{Y^\delta}' \times H}=\langle ({E_Y^\delta}' \ell,u_0), B_e^\delta w\rangle_{{Y^\delta}' \times H}$ ($w \in X^\delta$).
Using Theorem~\ref{thm:3}, we infer that 
$\|u^\delta\|_Y \leq \|u^\delta\|_{X^\delta} \leq \sqrt{c(\alpha)} \|B_e^\delta u^\delta\|_{{Y^\delta}' \times H} \leq \sqrt{c(\alpha)} \sqrt{\|\ell\|_{Y'}^2+\|u_0\|_H^2}$.

So, since $X^\delta$ is a finite-dimensional space, the set of these $u^\delta$ is precompact in $Y$, and, again by finite dimension, also in $X$.
Hence, also the corresponding set of solutions $\lambda_\delta=A_s^{-1}(\ell-B u^\delta)$ is precompact in $Y$. We conclude that for $Y^\delta  \supset X^\delta \cup Y^{\delta_i}$ and $i \rightarrow \infty$, it holds that $\lambda_\delta- \lambda^\delta \rightarrow 0$ in $Y$.

From $\|\lambda_\delta\|_Y=\|B(u-u^\delta)\|_{Y'}$, we have
$$
\|B_e(u-u^\delta)\|_{Y' \times H}^2-\big(\|\lambda^\delta\|_Y^2+\|u_0-\gamma_0  u^\delta\|_H^2\big)=\|\lambda_\delta-\lambda^\delta\|_Y^2  \rightarrow 0.
$$
Since, for any $Y^\delta$, $\|B_e(u-u^\delta)\|_{Y'}^2 \geq \frac{1}{c(\alpha)} \|u-u^\delta\|_X^2 \geq \frac{1}{c(\alpha)} \inf_{w \in X^\delta}\|u-w\|_X^2$,
we conclude that when $u \not\in X^\delta$, for $Y^\delta  \supset X^\delta \cup Y^{\delta_i}$ and $i$ large enough the right-hand side of \eqref{eq:pjotr} stays away from $0$, whereas its left-hand side converges to $0$, which completes the proof in this case.
Finally, if $u \in X^\delta$, then the unique solution of \eqref{eq:2}--\eqref{eq:3} is $(\lambda^\delta,u^\delta)=(0,u)$ whilst also $\lambda_\delta=0$. This means that \eqref{eq:pjotr} is valid for any $Y^\delta  \supset X^\delta$.
\end{proof}

To be able to verify \eqref{eq:pjotr}, we need an \emph{a posteriori error estimator} of the difference of $\lambda_\delta$ and its Galerkin approximation $\lambda^\delta\in Y^\delta$ in the $\|\bigcdot\|_Y$-norm.
The construction of such an estimator is the topic of the next section.  
 
 \section{A posteriori error estimator for $\|\lambda_\delta- \lambda^\delta\|_Y$} \label{sec:apost}
 \subsection{Specification of $H$, $V$, and $A$}
For some bounded Lipschitz domain $\Omega \subset \R^d$, let $H=L_2(\Omega)$, $V=H^1_0(\Omega)$, so that $Y=L_2(\Xi;H^1_0(\Omega))$. We set the Friedrichs constant 
\be  \label{eq:23}
 C_{F,\Omega}:=\sup_{0 \neq v \in H^1_0(\Omega)}\tfrac{\|v\|_{L_2(\Omega)}}{\|\nabla v\|_{L_2(\Omega)^d}}<\infty.
 \ee
Recalling that $\Sigma=\Xi \times \Omega$, we consider the case that for some $0<M=M^* \in L_\infty(\Sigma)^{d \times d}$ with $M^{-1} \in L_\infty(\Sigma)^{d \times d}$,
 \be \label{eq:22}
 (A_s w)(v)=\iint_\Sigma M \nabla_x w \cdot \nabla_x v \diff t \diff x \quad(w,v\in Y),
 \ee
 and that
 $$
 A_a \in \cL(X,L_2(\Sigma)).
 $$
 A further assumption on $A_a$ will be made in \eqref{eq:107}.
 
Recall that $\|\bigcdot\|_Y^2=(A_s\bigcdot)(\bigcdot)$, and that  $\lambda_\delta \in Y$ and $\lambda^\delta \in Y^\delta$ are the solutions of
\begin{align*} 
(A_s \lambda_\delta)(v)& = (\ell - B u^\delta)(v) \quad(v \in Y),\\ 
(A_s \lambda^\delta)(v)& = (\ell - B u^\delta)(v) \quad(v \in Y^\delta).
\end{align*}
Thanks to $X^\delta \subset Y^\delta$, for $\theta_\delta:=\lambda_\delta+u^\delta \in Y$ and $\theta^\delta:=\lambda^\delta+u^\delta \in Y^\delta$ it holds that 
\be \label{eq:103}
\lambda_\delta-\lambda^\delta=\theta_\delta-\theta^\delta,
\ee
 where, taking into account \eqref{eq:22}, with $B_a:=B-A_s=\diff_t+A_a$,
\begin{align} \label{eq:100}
\langle M \nabla_x \theta_\delta,\nabla_x v\rangle_{L_2(\Sigma)^d}& = (\ell - B_a u^\delta)(v) \quad(v \in Y),\\ \label{eq:101}
\langle M \nabla_x \theta^\delta,\nabla_x v\rangle_{L_2(\Sigma)^d}& = (\ell - B_a u^\delta)(v) \quad(v \in Y^\delta).
\end{align}

We consider $\ell \in L_2(\Sigma)$, and $(X^\delta)_{\delta \in \Delta}$ being a family of continuous finite element spaces on $\Sigma$ so that the forcing term $\ell - B_a u^\delta \in L_2(\Sigma)$.
  
 \subsection{Specification of $(X^\delta)_{\delta \in \Delta}$, $(Y^\delta)_{\delta \in \Delta}$, and an additional family $(\widecheck{Y}^\delta)_{\delta \in \Delta}$}
We will select families of subspaces of $X$ or $Y$ as finite element spaces w.r.t.~prismatic partitions of the space-time cylinder 
$\Sigma:=\Xi \times \Omega$.
 A (closed) \emph{prism} $P$ in $\overline{\Sigma}$ is a Cartesian product of an interval $I_P \subset \overline{\Xi}$ and a $d$-simplex $T_P$ in $\overline{\Omega}$.
For $p,q \in \N_0$, we set $Q_{p,q}(P):=P_p(I_P)\otimes P_q(T_P)$, and for a partition $\pria$ of $\overline{\Sigma}$ into prisms, we set $Q_{p,q}(\pria):=\{w\colon w|_P \in Q_{p,q}(P)\,(P \in \pria)\}$.

For a family of partitions $\bbP=(\pria^\delta)_{\delta \in \Delta\setminus\{\infty\}}$ of $\overline{\Sigma}$ into prisms and some $p \in \N$, we set
\be \label{eq:105}
X^\delta := C(\overline{\Sigma}) \cap L_2(\Xi;H_0^1(\Omega)) \cap Q_{p,p}(\pria^\delta),
\ee
being a subspace of $X$.
Concerning this choice, notice that when considering (uniformly) shape-regular prisms, as we will do, for general smooth solutions increasing the polynomial degree $p$ in either temporal or spatial direction cannot be expected to yield better rates than $\text{ndof}^{-\frac{p}{d+1}}$ in $X$. In other words, taking equal orders in time and space is the most reasonable choice.

To ensure that arising `data-oscillation terms' depend indeed exclusively on the data (specifically on $\ell$), and not on $u^\delta$, we add the assumption that
\be \label{eq:107}
A_a w^\delta \in Q_{p,p}(\pria^\delta)\quad(\pria^\delta \in \bbP,\,w^\delta \in X^\delta).
\ee

We consider an \emph{initial partition} $\pria^{\delta_0}$ that is the product of a partition of $\overline{\Xi}$ into subintervals, and a \emph{conforming} partition of $\overline{\Omega}$ into $d$-simplices that, for $d \geq 2$,  satisfies the \emph{matching condition} \cite{249.87} for the application of \emph{newest vertex bisection}\footnote{In \cite{64.588}, it was shown that the key results about newest vertex bisection are valid even without the matching condition. Using additional arguments, this might extend to the current setting of prismatic partitions.}.
Any prism $P$ in any partition $\pria^\delta \in \bbP$ will be created from a prism in the initial partition by (zero or more) recurrent \emph{splits} that we specify below. The number of these splits to create $P$ will be the generation $\gen(P)$ of $P$. If $\gen(P) \mod d \in \{0,\ldots, d-2\}$, then a split of $P$ consists of a (spatial) bisection of $T_P$ by cutting its \emph{refinement edge} determined by the newest vertex bisection rule, and if $\gen(P) \mod d=d-1$ it consists of such a spatial bisection \emph{combined with} a (temporal) bisection of $I_P$. 
The definition of our family $\bbP:=(\pria^\delta)_{\delta \in \Delta\setminus\{\infty\}}$ is completed by imposing the condition that any $\pria^\delta \in \bbP$ has to be \emph{spatially conforming} meaning that for a.e.~$t \in \overline{\Xi}$, $\tria_t^\delta:= \{T_P \colon P \in \pria^\delta,\,t \in I_P\}$ is a conforming simplicial partition of $\overline{\Omega}$.
This imposition of spatial conformity has the consequence that, for $d>1$, we cannot split prisms $P \in \pria^\delta$ on an individual basis. Instead, if we want to split $P \in \pria^\delta$, then we apply the following routine, which returns the coarsest refinement of $\pria^\delta$ inside $\bbP$ in which $P$ has been split (cf.~\cite[Algorithm 3]{64.588} for simplicial partitions):\vspace{1ex}

{\rm%
\begin{algotab}
\> \bf{re}\=\bf{fine}$(\pria^\delta,P)$\\
\>\>  Let $\mathcal{S}$ be the set of $P' \in \pria^\delta$ such that $|I_P \cap I_{P'}|>0$ and $T_{P'}$ contains\\
\>\>  the refinement edge $e$ of $T_P$.\\
\>\>  \texttt{if} \=$\exists P' \in \mathcal{S}$ for which the refinement edge of $T_P'$ is unequal to $e$ \texttt{then}\\
\>\>\>   \texttt{return} \bf{refine}$($\bf{refine}$(\pria^\delta,P'),P)$\\
\>\> \texttt{else} \\
\>\>\> \texttt{return} $\pria^\delta$ in which simultaneously all $P' \in \mathcal{S}$ have been split\\
\>\>  \texttt{endif} 
\end{algotab}}
see Figure~\ref{fig:3} for an illustration.
\begin{figure}[h]
\centering
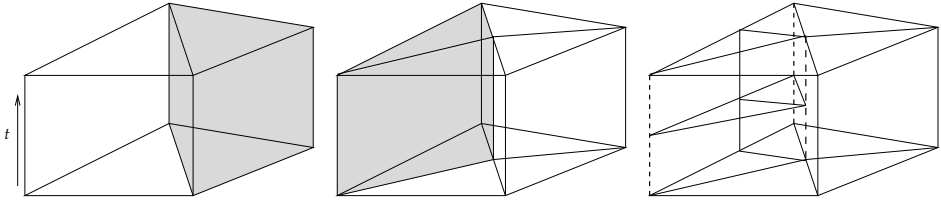
\caption{The result of two subsequent calls of refine$(\pria^\delta,P)$ for $d=2$ with the argument $P$ in grey. The midpoints of dashed lateral edges are `hanging'.}
\label{fig:3}
\end{figure}

It is well-known that with $h_{T}$ denoting the diameter of a $d$-simplex $T$, and $\rho_{T}$ the diameter of its largest inscribed ball, the newest vertex bisection procedure ensures that the uniform \emph{spatial} shape-regularity parameter
$$
\kappa_\bbP:=\sup_{\pria \in \bbP} \sup_{P \in \pria} \frac{h_{T_P}}{\rho_{T_P}}<\infty,
$$
only dependent on the initial partition of $\overline{\Omega}$ into $d$-simplices.

\begin{remark} The definition of the \emph{splits} guarantees that consequently the prisms from the partitions in $\bbP$ are uniformly space-time shape-regular. 
Alternatively, as in \cite{64.587} one could aim for parabolically scaled prisms by replacing in every $d$th split the temporal bisection by a division of the time interval into 4 subintervals (or for $d=2$, by combining every spatial bisection by a temporal bisection).  In that case all findings in this work, including those from the appendix, will remain true.
A further option would be to adaptively select between a spatial or a temporal refinement of a prism.
\end{remark}

\begin{remark} \label{rem:one-dimensional-grading}
For $d=1$, any prismatic partition of $\overline{\Sigma}$ is spatially conforming.
Nevertheless, in order to obtain prismatic partitions that, as in the case $d \geq 2$, are \emph{spatially uniformly graded}, one can impose the condition that for $P,P' \in \pria^\delta \in \bbP$ with $P \cap P' \neq \emptyset$ and $I_P \subset I_{P'}$, it holds that  $\gen(P) \leq \gen(P')+1$.
An obvious adaptation of the above refine routine will ensure this.
\end{remark}

We equip the index set $\Delta\setminus\{\infty\}$ with a \emph{partial ordering} by defining
$$
\underline{\delta} \succeq \delta \Longleftrightarrow \pria^{\underline{\delta}} \text{ is a refinement of } \pria^\delta.
$$
For $\delta \in \Delta\setminus\{\infty\}$, and some $\Delta\setminus\{\infty\} \ni \underline{\delta} \succeq \delta$, we set
\be \label{eq:106}
Y^\delta:=L_2(\Xi;H_0^1(\Omega)) \cap Q_{p,p}(\pria^{\underline{\delta}}).
\ee
Here we include the possibility that $\underline{\delta} \neq \delta$ to accommodate a possible enlargement of the test space $Y^\delta$ in the case that the a posteriori condition \eqref{eq:pjotr} for quasi-optimality of $u^\delta$ is not satisfied for $\underline{\delta} = \delta$.
\medskip

Unfortunately, with the above definition of $Y^\delta$, we are not able to construct a finite-element-type a posteriori error estimator for $\|\theta_\delta-\theta^\delta\|_Y$ that, modulo data-oscillation terms, is efficient and reliable. The difficulty is the purely anisotropic diffusion operator $-\divv_x M \nabla_x$ in combination with a partition $\pria^{\underline{\delta}}$ that is generally not decomposable into time-slabs.

Therefore, let $\widecheck{\pria}^{\underline{\delta}}$ be the coarsest prismatic refinement of  $\pria^{\underline{\delta}}$ that allows such a time-slab decomposition, meaning that it contains no prisms that have lateral edges whose midpoints are hanging (cf. Figure~\ref{fig:3}). Notice that generally $\widecheck{\pria}^{\underline{\delta}} \not\in \bbP$ because it is created from $\pria^{\underline{\delta}}$ by temporal cuts only.
In Figure~\ref{fig:4}, an illustration is given of its construction in the $d=1$ case. Notice that the prisms from the partitions $\widecheck{\pria}^{\underline{\delta}}$ are \emph{not} uniformly shape-regular.
\begin{figure}[h]
\centering
\includegraphics{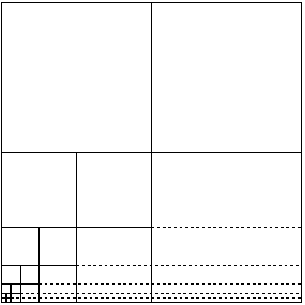}
\caption{An example of $\pria^{\underline{\delta}} \in \bbP$ and its refinement  $\widecheck{\pria}^{\underline{\delta}}$ for $d=1$.}
\label{fig:4}
\end{figure}

Setting
$$
\widecheck{Y}^\delta:=L_2(\Xi;H_0^1(\Omega)) \cap Q_{p,p}(\widecheck{\pria}^{\underline{\delta}}),
$$
similarly to $\theta_\delta$ and $\theta^\delta$ from \eqref{eq:100}-\eqref{eq:101}, we define $\widecheck{\theta}^\delta \in \widecheck{Y}^\delta$ by 
\be \label{eq:102}
\langle M \nabla_x \widecheck{\theta}^\delta,\nabla_x v\rangle_{L_2(\Sigma)^d} = (\ell - B_a u^\delta)(v) \quad(v \in \widecheck{Y}^\delta).
\ee
Using that $Y^\delta \subset \widecheck{Y}^\delta$,  we have
\be \label{eq:104}
\big(\|\lambda_\delta-\lambda^\delta\|_Y^2=\!\!\big)\,\,\|\theta_\delta-\theta^\delta\|_Y^2=\|\theta_\delta-\widecheck{\theta}^\delta\|_Y^2+\|\widecheck{\theta}^\delta-\theta^\delta\|_Y^2.
\ee
By solving \eqref{eq:102}, we can compute $\|\widecheck{\theta}^\delta-\theta^\delta\|_Y^2$. In the next subsection we develop an a posteriori error estimator for 
$\|\theta_\delta-\widecheck{\theta}^\delta\|_Y^2$.

\begin{remark} \label{rem:slabbifiedY}
Computing $\widecheck{\theta}^\delta$ requires solving independent elliptic equations for each time slab. With optimal preconditioners, $\widecheck{\theta}^\delta$ can be approximated within a sufficiently small tolerance in $\eqsim \# \widecheck{\pria}^{\underline{\delta}}$ operations. Unfortunately, $\# \widecheck{\pria}^{\underline{\delta}} / \# \pria^{\underline{\delta}}$ cannot be uniformly bounded.
Exploiting the independence of the elliptic equations, the effort can essentially be reduced to $\max_{t\in\overline\Xi} \#\widecheck{\tria}_t^\delta$ per thread on massively parallel machines.
Here, $\widecheck{\tria}_t^\delta$ denotes as before the spatial mesh defined for a.e.\ $t$.
\end{remark}

\subsection{An a posteriori error estimator for $\|\theta_\delta-\widecheck{\theta}^\delta\|_Y$}
On each time-slab, the operator $A_s$ is the tensor product of the identity operator in time and an elliptic operator in space.
Aiming at an estimator that, modulo higher order oscillation terms, is efficient and reliable, and that provides an upper bound on  $\|\theta_\delta-\widecheck{\theta}^\delta\|_Y$ whose leading term does not incorporate a generic unknown constant, we use the technique of equilibrated flux estimators \cite{33,70.8,58.95}.

As with partitions from $\bbP$, any $P \in \widecheck{\pria}^{\underline{\delta}}$ is of the form $P=I_P \times T_P$ for an interval $I_P \subset \overline{\Xi}$ and a $d$-simplex $T_P \subset \overline{\Omega}$, which is uniformly shape-regular. Specifically, $\frac{h_{T_P}}{\rho_{T_P}} \leq \kappa_\bbP$.
For a.e.~$t \in \overline{\Xi}$, $\widecheck{\tria}_t^{\underline\delta}:= \{T_P \colon P \in \widecheck{\pria}^{\underline\delta},\,t \in I_P\}$ is a conforming simplicial partition of $\overline{\Omega}$.
With $N(T)$ denoting the set of vertices of a $d$-simplex $T$, we define the set of lateral edges $\widecheck{\cE}^{\underline{\delta}}$ of $\widecheck{\pria}^{\underline{\delta}}$ by
$$
\widecheck{\cE}^{\underline{\delta}} := \{I_P \times \{\nu\} \colon P \in \widecheck{\pria}^{\underline{\delta}},\,\nu \in N(T_P)\}.
$$
We decompose $\widecheck{\cE}^{\underline{\delta}} = \widecheck{\cE}_{\text{int}}^{\underline{\delta}} \dot\cup \widecheck{\cE}_{\text{bdr}}^{\underline{\delta}}$, where
$\widecheck{\cE}_{\text{bdr}}^{\underline{\delta}}:=\{e \in \widecheck{\cE}^{\underline{\delta}}\colon e \subset \overline{\Xi} \times \partial\Omega\}$.

Any $e\in \widecheck{\cE}^{\underline{\delta}}$ is of the form $I_e \times \{\nu_e\}$, where for a.e.~$t \in I_e$, $\nu_e$ is a vertex of the conforming simplicial partition $\widecheck\tria^{\underline{\delta}}_t$ of $\overline{\Omega}$. 
The local patch $\{T \in \widecheck\tria_t^{\underline{\delta}}\colon \nu_e \in N(T)\}$ is independent of those $t$, 
and we therefore will denote it by $\widecheck\tria^{\underline{\delta}}_e$.
Let $\phi_e \in C(\overline{\Omega})$ denote the corresponding continuous piecewise linear nodal `hat'-function, i.e., $\phi_e(\nu_e)=1$ and $\phi_e=0$ outside $\omega_e:=\bigcup \{T \in \widecheck\tria_e^{\underline{\delta}}\}$.
Then, with $\psi_e:=\chi_{I_e} \otimes \phi_e$,
we have
\be \label{eq:unity}
\sum_{e \in \widecheck{\cE}^{\underline{\delta}}} \psi_e =1\quad \text{on } \overline{\Sigma},
\ee
where
$$
\supp \psi_e = \Sigma_e:=I_e \times \omega_e.
$$
The partition of unity $\{\psi_e\colon e \in \widecheck{\cE}^{\underline{\delta}}\}$ will be used to localize the residual of the approximation $\widecheck{\theta}^\delta$ for $\theta_\delta$.
For $e \in \widecheck{\cE}_{\rm int}^{\underline{\delta}}$ and $q \in P_p(I_e)$, it holds that $\psi_e(q \otimes \chi_{\omega_e}) \in Y^\delta$.

For $e \in \widecheck{\cE}^{\underline{\delta}}$, define\footnote{For $d=1$, $ \RT_{p}(T)$ should be read as $P_{p+1}(T)$.}
\begin{align*}
&\RT_{p,0}(\widecheck\tria^{\underline{\delta}}_e) := \prod_{T \in \widecheck\tria^{\underline{\delta}}_e} \RT_{p}(T)  \\
&\qquad \cap \left\{\begin{array}{ll}  \big\{\sigma \in H(\divv;\omega_e)\colon \sigma \cdot {\bf n}_{\omega_e}=0 \text{ on } \partial\omega_e\big\} & \text{when } e \in \widecheck{\cE}^{\underline{\delta}}_{\text{int}},\\
\big\{\sigma \in H(\divv;\omega_e)\colon \sigma \cdot {\bf n}_{\omega_e}=0 \text{ on } \partial\omega_e\setminus \partial\Omega\big\} & \text{when }  e \in \widecheck{\cE}^{\underline{\delta}}_{\text{bdr}}, \end{array}\right.
\end{align*}
and
$$
P_{p,*}(\widecheck\tria^{\underline{\delta}}_e):=\left\{\begin{array}{ll}  \big\{q \in \prod_{T \in \widecheck\tria^{\underline{\delta}}_e} P_{p}(T) \colon \int_{\omega_e} q \diff x=0\big\} & \text{when } e \in \widecheck{\cE}^{\underline{\delta}}_{\text{int}},\\
\prod_{T \in \widecheck\tria^{\underline{\delta}}_e} P_{p}(T)& \text{when }  e \in \widecheck{\cE}^{\underline{\delta}}_{\text{bdr}}. \end{array}\right.
$$
Let 
 $(\sigma_e^\delta,z_{e}^\delta) \in (P_p(I_e) \otimes \RT_{p,0}(\widecheck\tria^{\underline{\delta}}_e)) \times (P_p(I_e) \otimes P_{p,*}(\widecheck\tria^{\underline{\delta}}_e))$ solve
 \be \label{eq:16}
 \left\{
 \begin{array}{@{}r@{}c@{}l@{}}
 \langle M^{-1}\sigma_e^\delta, \widetilde{\sigma}\rangle_{L_2(\Sigma_e)^d}\!+\! \langle \divv_x \widetilde{\sigma}, z_{e}^\delta \rangle_{L_2(\Sigma_e)} & \,=\, & - \langle \psi_e \nabla_x \widecheck{\theta}^\delta, \widetilde{\sigma}\rangle_{L_2(\Sigma_e)^d} ,\\
 \langle \divv_x \sigma_e^\delta, \widetilde{z} \rangle_{L_2(\Sigma_e)} &\,= \,& \langle \psi_e (\ell-B_a u^\delta) - \nabla_x \psi_e\cdot M \nabla_x \widecheck{\theta}^\delta, \widetilde{z}\rangle_{L_2(\Sigma_e)}
 \end{array}
 \right.
 \ee
for all  $(\widetilde{\sigma},\widetilde{z}) \in (P_p(I_e) \otimes\RT_{p,0}(\widecheck\tria^{\underline{\delta}}_e)) \times (P_p(I_e) \otimes P_{p,*}(\widecheck\tria^{\underline{\delta}}_e))$. So
$\sigma_e^\delta$ is the minimizer of $ \|M^{-\frac12} \widetilde{\sigma}+\psi_e M^{\frac12} \nabla_x \widecheck{\theta}^\delta\|_{L_2(\Sigma_e)^d}$ over $\widetilde{\sigma}\in P_p(I_e) \otimes \RT_{p,0}(\widecheck\tria^{\underline{\delta}}_e)$, under the constraint that $\big(\divv_x \widetilde{\sigma}-\psi_e (\ell-B_a u^\delta) - \nabla_x \psi_e\cdot M \nabla_x \widecheck{\theta}^\delta \big)\perp_{L_2(\Sigma_e)}P_p(I_e) \otimes P_{p,*}(\widecheck\tria^{\underline{\delta}}_e)$.

 \begin{theorem}[Reliability] \label{thm:reliability}
With $\sigma^\delta:=\sum_{e \in \widecheck{\cE}^{\underline{\delta}}} \sigma_e^\delta$, it holds that 
$$
\|\theta_\delta -  \widecheck{\theta}^\delta \|_Y
\leq
\|M^{-\frac12} \sigma^\delta+M^{\frac12} \nabla_x  \widecheck{\theta}^\delta\|_{L_2(\Sigma)^d}+\osc(\ell,\widecheck{\pria}^{\underline{\delta}}),
$$
where
\begin{align*}
\osc(&\ell,\widecheck{\pria}^{\underline{\delta}}):=
\Big(C_{F,\Omega}\sqrt{\sum_{P \in \widecheck{\pria}^{\underline{\delta}}}  \inf_{q \in P_p(I_P) \otimes L_2(T_P)} \|\ell-q\|_{L_2(P)}^2}\\
&+\tfrac{1}{\pi} \sqrt{\sum_{P \in \widecheck{\pria}^{\underline{\delta}}} 
{\|\lambda_{\min}(M)^{-1}\|_{L_\infty(P)}}
\diam(T_P)^2  \inf_{q \in L_2(I_P) \otimes P_p(T_P)} \|\ell-q\|^2_{L_2(P)}} \,\Big).
\end{align*}
 \end{theorem}
 
 \begin{remark} For smooth $\ell$, $\osc(\ell,\widecheck{\pria}^{\underline{\delta}})$ is of higher order than $\|\theta_\delta -  \widecheck{\theta}^\delta \|_Y$.
 \end{remark}
 
\begin{proof}
For $v \in Y$, we set
$$
r^\delta(v):=
\langle \ell - B_a u^\delta,v\rangle_{L_2(\Sigma)}-\langle M \nabla_x \widecheck{\theta}^\delta, \nabla_x v \rangle_{L_2(\Sigma)^d}.
$$
Integration by parts shows that
\begin{align*}
 \|\theta_\delta-\widecheck{\theta}^\delta\|_Y&=\sup_{0 \neq v \in Y}   \frac{r^\delta(v)}{\|M^{\frac12}\nabla_x v\|_{L_2(\Sigma)^d}}\\
 =&\sup_{0 \neq v \in Y}   \frac{\langle \ell - B_a u^\delta-\divv_x \sigma^\delta,v\rangle_{L_2(\Sigma)}-\langle \sigma^\delta+M \nabla_x \widecheck{\theta}^\delta, \nabla_x v \rangle_{L_2(\Sigma)^d}}
{\|M^{\frac12}\nabla_x v\|_{L_2(\Sigma)^d}}\\
 \leq&  \|M^{-\frac12} \sigma^\delta+M^{\frac12} \nabla_x \widecheck{\theta}^\delta\|_{L_2(\Sigma)^d}+\sup_{0 \neq v \in Y}  
\frac{ \langle \ell - B_a u^\delta-\divv_x \sigma^\delta,v\rangle_{L_2(\Sigma)}}{\|M^{\frac12} \nabla_x v\|_{L_2(\Sigma)^d}}.
\end{align*}
It remains to show that the second term can be bounded by $\osc(\ell,\widecheck{\pria}^{\underline{\delta}})$.

For $e \in \widecheck{\cE}^{\underline{\delta}}$, and $v$ such that $\psi_e v \in Y$, we set
\be
\begin{split} \label{eq:g4}
r_e^\delta(v):\!&=r^\delta(\psi_e v)\\
& =
\langle \psi_e(\ell - B_a u^\delta) -\nabla_x \psi_e \cdot M \nabla_x \widecheck{\theta}^\delta,v\rangle_{L_2(\Sigma_e)}-\langle \psi_e M \nabla_x \widecheck{\theta}^\delta, \nabla_x v \rangle_{L_2(\Sigma_e)^d}.
\end{split}
\ee
From \eqref{eq:unity}, for $v \in L_2(\Sigma)$ we have
\be \label{eq:g10}
\begin{split}
\langle \ell - B_a u^\delta-\divv_x &\sigma^\delta,v\rangle_{L_2(\Sigma)}=\\
& \sum_{e  \in \widecheck{\cE}^{\underline{\delta}}} \underbrace{\langle \psi_e(\ell - B_a u^\delta) -\nabla_x \psi_e \cdot M \nabla_x \widecheck{\theta}^\delta-\divv_x \sigma_{e}^\delta,v\rangle_{L_2(\Sigma_e)}}_{d_e^\delta(v):=}.
\end{split}
\ee
Both $r_e^\delta(v)$ and $d_e^\delta(v)$ depend only on $v|_{\Sigma_e}$.

For any $e \in \widecheck{\cE}^{\underline{\delta}}$, the definition of $\sigma_e^\delta$ shows that
\be \label{eq:g11}
d_e^\delta|_{P_p(I_e) \otimes P_{p,*}(\widecheck\tria^{\underline{\delta}}_e)}=0,
\ee
i.e., for any $e \in \widecheck{\cE}_{\rm bdr}^{\underline{\delta}}$ it holds that $d_e^\delta|_{P_p(I_e) \otimes \prod_{T \in \widecheck\tria_e^\delta} P_{p}(T)}=0$.

For any $e \in \widecheck{\cE}_{\rm int}^{\underline{\delta}}$ and $q \in P_p(\Xi)$,
$$
\langle \psi_e M \nabla_x \widecheck{\theta}^\delta, \nabla_x (q \otimes \chi_{\omega_e}) \rangle_{L_2(\Sigma_e)^d}=0=\langle \divv_x \sigma_{e}^\delta,q \otimes \chi_{\omega_e}\rangle_{L_2(\Sigma_e)},
$$
so that
$$
d_e^\delta(q \otimes \chi_{\omega_e})=r_e^\delta(q \otimes \chi_{\omega_e})=r^\delta(\psi_e (q \otimes \chi_{\omega_e}))=0
$$
by $\psi_e (q \otimes \chi_{\omega_e})\in \widecheck Y^\delta$ and Galerkin orthogonality. In combination with \eqref{eq:g11}, we conclude that
$$
d_e^\delta|_{P_p(I_e) \otimes \prod_{T \in \widecheck\tria_e^\delta} P_{p}(T)}=0 \quad(e \in \widecheck{\cE}^{\underline{\delta}}).
$$

By \eqref{eq:g10}, we conclude that for $v|_P \in Q_{p,p}(P)$, $P\in\widecheck\pria^{\underline\delta}$,
\begin{align*}
\langle \ell-B_a u^\delta-\divv_x \sigma^\delta,v\rangle_{L_2(P)}&=
\langle \ell-B_a u^\delta-\divv_x \sigma^\delta,\chi_P v\rangle_{L_2(\Sigma)}\\
&=\sum_{e \in \widecheck{\cE}^{\underline{\delta}}} d_e^\delta( \chi_P v)=0,
\end{align*}
because for $e \in \widecheck{\cE}^{\underline{\delta}}$ either $|P\cap \Sigma_e|=0$ or $P \in I_e  \times \widecheck\tria_e^{\underline{\delta}}$. 

With $\Pi_{I_P}^p$ 
and $\Pi^{p}_{T_P}$ denoting the $L_2$-orthogonal projectors onto $P_p(I_P)$ and $P_{p}(T_P)$, respectively, for $v \in Y$ we conclude that
$$
\langle \ell-B_a u^\delta-\divv_x \sigma^\delta,v\rangle_{L_2(P)}=
\langle (\identity-\Pi_{I_P}^p \otimes \Pi^{p}_{T_P})(\ell-B_a u^\delta-\divv_x \sigma^\delta),v\rangle_{L_2(P)}.
$$

Using that for $P \in \widecheck\pria^{\underline{\delta}}$, it holds that
$(B_a u^\delta+\divv_x \sigma^\delta)|_P \in Q_{p,p}(P)$ by~\eqref{eq:107}, for $v \in Y$ we have
\be \label{eq:g1}
\begin{split}
|\langle (\identity&-\Pi_{I_P}^p \otimes \Pi^{p}_{T_P})(\ell-B_a u^\delta-\divv_x \sigma^\delta),v\rangle_{L_2(P)}|\\
&=
|\langle (\identity-\Pi_{I_P}^p \otimes \Pi^{p}_{T_P})\ell,v\rangle_{L_2(P)}|\\
&=\big|\langle \big((\identity-\Pi_{I_P}^p)\otimes \identity+\Pi_{I_P}^p\otimes (\identity-\Pi_{T_P}^{p})\big)  \ell,v \rangle_{L_2(P)}\big|\\
&\leq  \|v\|_{L_2(P)} \inf_{q \in P_p(I_P) \otimes L_2(T_P)} \|\ell-q\|_{L_2(P)}+\\
 &\quad \,\,\,\big|\langle (\Pi_{I_P}^p\otimes (\identity-\Pi_{T_P}^{p}))  (\identity \otimes (\identity-\Pi_{T_P}^{p})) \ell, (\identity \otimes (\identity-\Pi_{T_P}^{0})) v \rangle_{L_2(P)}\big|\\
&\leq  \|v\|_{L_2(P)} \inf_{q \in P_p(I_P) \otimes L_2(T_P)} \|\ell-q\|_{L_2(P)}+\\
& \quad\,\,\,\tfrac{1}{\pi} \diam(T_P)  \inf_{q \in L_2(I_P) \otimes P_p(T_P)} \|\ell-q\|_{L_2(P)} \|\nabla_x v\|_{L_2(P)^d},
\end{split}
\ee
where we used 
$\|(\Pi_{I_P}^p\otimes (\identity-\Pi_{T_P}^{p})) ( \bigcdot) \|_{L_2(P)} \leq \|\bigcdot\|_{L_2(P)}$ and the Poincar\'{e} inequality on convex domains (\cite{19.895}).

By applications of Cauchy--Schwarz inequalities, the Friedrichs inequality \eqref{eq:23}, and $\|\cdot\|_{L_2(P)^d} \leq \|\lambda_{\min}(M)^{-\frac12}\|_{L_\infty(P)}  \|M^{\frac12} \cdot\|_{L_2(P)^d}$, we conclude that
$$
\sup_{0 \neq v \in Y}  
\frac{ \langle \ell - B_a u^\delta-\divv_x \sigma^\delta,v\rangle_{L_2(\Sigma)}}{\|M^{\frac12} \nabla_x v\|_{L_2(\Sigma)^d}} \leq \osc(\ell,\widecheck{\pria}^{\underline{\delta}})
$$
which completes the proof.
\end{proof}
 
To show efficiency, for $e \in \widecheck{\cE}^{\underline{\delta}}$ we set 
$$
H_*^1(\omega_e):=\left\{
\begin{array}{ll}
\{v \in H^1(\omega_e)\colon\int_{\omega_e} v \diff x=0 \} &  \text{when } e \in \widecheck{\cE}^{\underline{\delta}}_{\text{int}},\\
\{v \in H^1(\omega_e)\colon v=0 \text{ on } \partial\omega_e \cap \partial \Omega\} & \text{when } e \in \widecheck{\cE}^{\underline{\delta}}_{\text{bdr}},
\end{array}
\right.
$$
and equip it with $\|\nabla\bigcdot\|_{L_2(\omega_e)^d}$. There exists a constant $C_{PF}$, depending only on the uniform spatial shape-regularity parameter
$\kappa_\bbP$, see \cite[p.1066 + references cited there]{70.8}, such that, either by Poincar\'e's inequality for $e \in \widecheck{\cE}^{\underline{\delta}}_{\text{int}}$, or by Friedrichs' inequality for $e \in \widecheck{\cE}^{\underline{\delta}}_{\text{bdr}}$, 
\be \label{eq:52}
\|\bigcdot\|_{L_2(\omega_e)} \leq C_{PF} \diam(\omega_e) \|\nabla \bigcdot\|_{L_2(\omega_e)^d} \quad\text{on }H_*^1(\omega_e).
\ee
 
\begin{theorem}[local efficiency] \label{thm:efficiency}
Let $d \in \{1,2,3\}$.
There exist constants $C_1, C_2$, only dependent on the uniform spatial shape-regularity parameter $\kappa_\bbP$, such that
for $M$ being piecewise constant w.r.t.~$\widecheck{\pria}^{\underline{\delta}}$, it holds for all $P\in\widecheck\pria^{\underline\delta}$ that
\begin{align*}
& \|M^{-\frac12} \sigma^\delta+M^{\frac12} \nabla_x \widecheck{\theta}^\delta\|_{L_2(P)^d} \leq \\
& C_1 \hspace*{-1.2em}\sum_{\{e \in \widecheck{\cE}^{\underline{\delta}}\colon |\Sigma_e \cap P|>0 \}}\hspace*{-1.5em}
\|\lambda_{\min}(M)^{-\frac12}\|_{L_\infty(\Sigma_e)}  \Big[
(1\!+\!C_2) \|\lambda_{\max}(M)^{\frac12}\|_{L_\infty(\Sigma_e)} \|M^{\frac12} \nabla_x(\theta_\delta\!-\!\widecheck{\theta}^\delta)\|_{L_2(\Sigma_e)^d}\\
 &\hspace*{13em}+
  \tfrac{1}{\pi} \sqrt{ \sum_{T \in \widecheck\tria_e^{\underline{\delta}}} \diam(T)^2\inf_{q \in L_2(I_e)\otimes P_{p}(T)}\|\ell-q\|_{L_2(I_e \times T)} ^2}\,\,
 \Big].
 \end{align*}
\end{theorem}

\begin{proof} Recall that $\sigma^\delta=\sum_{e \in \widecheck{\cE}^{\underline{\delta}}} \sigma_e^\delta$, where $\sigma_e^\delta$ minimizes $ \|M^{-\frac12} \widetilde{\sigma}+\psi_e M^{\frac12} \nabla_x \widecheck{\theta}^\delta\|_{L_2(\Sigma_e)^d}$ over $\widetilde{\sigma}\in P_p(I_e) \otimes \RT_{p,0}(\widecheck\tria^{\underline{\delta}}_e)$ with $\big(\divv_x \widetilde{\sigma}-\psi_e (\ell-B_a u^\delta) - \nabla_x \psi_e\cdot M \nabla_x \widecheck{\theta}^\delta\big)\perp_{L_2(\Sigma_e)}P_p(I_e) \otimes P_{p,*}(\widecheck\tria^{\underline{\delta}}_e)$.
Let $\widehat{\sigma}_e^\delta$ be the minimizer of $\|\widetilde{\sigma}+\psi_e M \nabla_x \widecheck{\theta}^\delta\|_{L_2(\Sigma_e)^d}$ for 
 $\widetilde{\sigma}$ running over the same set. We estimate
\be \label{eq:g6}
\begin{split}
\hspace*{-2em} \|M^{-\frac12} \sigma^\delta+&M^{\frac12} \nabla_x \widecheck{\theta}^\delta\|_{L_2(P)^d} \leq
 \sum_{\{e \in \widecheck{\cE}^{\underline{\delta}}\colon |\Sigma_e \cap P|>0 \}} \|M^{-\frac12} \sigma_e^\delta+\psi_e M^{\frac12} \nabla_x \widecheck{\theta}^\delta\|_{L_2(\Sigma_e)^d}\\
&\leq
 \sum_{\{e \in \widecheck{\cE}^{\underline{\delta}}\colon |\Sigma_e \cap P|>0 \}} \|M^{-\frac12} \widehat{\sigma}_e^\delta+ \psi_e M^{\frac12}\nabla_x \widecheck{\theta}^\delta\|_{L_2(\Sigma_e)^d}\\
 &\leq
 \sum_{\{e \in \widecheck{\cE}^{\underline{\delta}}\colon |\Sigma_e \cap P|>0 \}}  \|\lambda_{\min}(M)^{-\frac12}\|_{L_\infty(\Sigma_e)}
 \|\widehat{\sigma}_e^\delta+ \psi_e M\nabla_x \widecheck{\theta}^\delta\|_{L_2(\Sigma_e)^d}.
\end{split}
\ee
By comparing to \eqref{eq:16}, we infer that $\widehat{\sigma}_e^\delta$ is the first component of the pair $(\widehat{\sigma}_e^\delta,z_{e}^\delta) \in (P_p(I_e) \otimes \RT_{p,0}(\widecheck\tria^{\underline{\delta}}_e)) \times (P_p(I_e) \otimes P_{p,*}(\widecheck\tria^{\underline{\delta}}_e))$ that solves
$$
 \left\{
 \begin{array}{@{}r@{}c@{}l@{}}
 \langle \widehat{\sigma}_e^\delta, \widetilde{\sigma}\rangle_{L_2(\Sigma_e)^d}\!+\! \langle \divv_x \widetilde{\sigma}, z_{e}^\delta \rangle_{L_2(\Sigma_e)} & \,=\, & - \langle \psi_e M \nabla_x \widecheck{\theta}^\delta, \widetilde{\sigma}\rangle_{L_2(\Sigma_e)^d} ,\\
 \langle \divv_x \widehat{\sigma}_e^\delta, \widetilde{z} \rangle_{L_2(\Sigma_e)} &\,= \,& \langle \psi_e (\ell-B_a u^\delta) - \nabla_x \psi_e\cdot M \nabla_x \widecheck{\theta}^\delta, \widetilde{z}\rangle_{L_2(\Sigma_e)}
 \end{array}
 \right.
$$
for all  $(\widetilde{\sigma},\widetilde{z}) \in (P_p(I_e) \otimes\RT_{p,0}(\widecheck\tria^{\underline{\delta}}_e)) \times (P_p(I_e) \otimes P_{p,*}(\widecheck\tria^{\underline{\delta}}_e))$. 

Let $(\zeta_{e,i})_{i \geq 0}$ be such that for any $n$ its first $n+1$ elements form an $L_2(I_e)$-orthonormal basis for $P_n(I_e)$.
We set 
$$
\widehat{\sigma}_{e,i}^\delta:=\int_{I_e} \widehat{\sigma}_{e}^\delta(t,\bigcdot) \zeta_{e,i}(t)\diff t,\,\,
\widecheck{\theta}_{e,i}^\delta:=\int_{I_e} \widecheck{\theta}_{e}^\delta(t,\bigcdot) \zeta_{e,i}(t)\diff t,\,\,
g_{e,i}^\delta:=\int_{I_e} (\ell-B_a u^\delta)(t,\bigcdot) \zeta_{e,i}(t)\diff t.
$$
Notice that for $i>p$, $\widehat{\sigma}_{e,i}^\delta=\widecheck{\theta}_{e,i}^\delta=0$ and $g_{e,i}^\delta=\int_{I_e} \ell(t,\bigcdot) \zeta_{e,i}(t)\diff t$ by \eqref{eq:107}.
We infer that for $0 \leq i \leq p$, $(\widehat{\sigma}_{e,i}^\delta,z_{e,i}^\delta) \in \RT_{p,0}(\widecheck\tria_e^{\underline{\delta}}) \times P_{p,*}(\widecheck\tria_e^{\underline{\delta}})$ solves
 \be \label{eq:72}
\left\{
 \begin{array}{@{}r@{}c@{}l@{}}
 \langle \widehat{\sigma}_{e,i}^\delta, \widetilde{\sigma}\rangle_{L_2(\omega_e)^d}\!+\! \langle \divv \widetilde{\sigma}, z_{e,i}^\delta \rangle_{L_2(\omega_e)} & \,=\, & - \langle \phi_e M \nabla \widecheck{\theta}_{e,i}^\delta, \widetilde{\sigma}\rangle_{L_2(\omega_e)^d} ,\\
 \langle \divv \widehat{\sigma}_{e,i}^\delta, \widetilde{z} \rangle_{L_2(\omega_e)} &\,= \,& \langle \phi_e 
g_{e,i}^\delta - \nabla \phi_e\cdot M \nabla \widecheck{\theta}_{e,i}^\delta, \widetilde{z}\rangle_{L_2(\omega_e)}
 \end{array}
 \right.
 \ee
for all $(\widetilde{\sigma},\widetilde{z}) \in \RT_{p,0}(\widecheck\tria_e^{\underline{\delta}}) \times P_{p,*}(\widecheck\tria_e^{\underline{\delta}})$.
Here, $M$ means $M(t,\cdot)$ for some $t$ in the interior of $I_e$ by abuse of notation.

For $0 \leq i \leq p$ and $v\in H_*^1(\omega_e)$, let 
$$
\widecheck{r}_{e,i}^\delta(v):=\sum_{T \in \widecheck\tria_e^{\underline{\delta}}}\langle \Pi^{p}_{T} (\phi_e g_{e,i}^\delta),  v \rangle_{L_2(T)}-\langle M \nabla \widecheck{\theta}_{e,i}^\delta,\nabla (\phi_e v)\rangle_{L_2(\omega_e)^d}.
$$
A celebrated result from \cite{33} for $d=2$ (cf.~\cite[Cor.~3.16]{70.8}) or \cite{70.9} for $d=3$ shows that for some constant
 $C_1>0$, only dependent on $\kappa_\bbP$ and thus in particular \emph{not} on the polynomial degree $p$, it holds that\footnote{A direct calculation shows that for $d=1$, \eqref{eq:72} has only one solution, which satisfies \eqref{eq:g5} with $C_1=1$.}
\be \label{eq:g5}
\|\widehat{\sigma}_{e,i}^\delta+\phi_e M \nabla \widecheck{\theta}_{e,i}^\delta\|_{L_2(\omega_e)^d} \leq C_1 \|\widecheck{r}_{e,i}^\delta\|_{H_*^1(\omega_e)'},
\ee
and so with $\widecheck{r}_{e}^\delta:=\sum_{i=0}^p \zeta_{e,i} \otimes \widecheck{r}_{e,i}^\delta$,
\be \label{eq:g7}
\|\sigma_e^\delta+\psi_e M \nabla_x \widecheck{\theta}_{e}^\delta\|_{L_2(\Sigma_e)^d} \leq C_1 \|\widecheck{r}_{e}^\delta\|_{P_p(I_e)\otimes H_*^1(\omega_e)'}.
\ee

The definition of $r^\delta_e(\bigcdot)=r^\delta(\psi_e \bigcdot)$ from \eqref{eq:g4}, and that of $g_{e,i}^\delta$ given above show that for $v \in P_p(I_e)\otimes H_*^1(\omega_e)$,
\be \label{eq:g8}
\begin{split}
|r_e^\delta(v)&-\widecheck{r}_{e}^\delta(v)|=\Big|\sum_{T \in \widecheck\tria_e^{\underline{\delta}}} \langle (\identity - \Pi_{I_e}^p\otimes \Pi_T^{p} ) (\psi_e \ell),v\rangle_{L_2(I_e \times T)}\Big|\\
=&\Big|\sum_{T \in \widecheck\tria_e^{\underline{\delta}}} \langle \big((\identity-\Pi_{I_e}^p)\otimes \identity+\Pi_{I_e}^p\otimes (\identity-\Pi_T^{p})\big) (\psi_e \ell ),v \rangle_{L_2(I_e \times T)}\Big|\\
=&\Big|\sum_{T \in \widecheck\tria_e^{\underline{\delta}}} \langle \big(\Pi_{I_e}^p\otimes (\identity-\Pi_T^{p})\big) (\psi_e \ell ),v \rangle_{L_2(I_e \times T)}\Big|\\
\leq & \tfrac{1}{\pi} \sqrt{\sum_{T \in \widecheck\tria_e^{\underline{\delta}}} \diam(T)^2  \inf_{q \in L_2(I_e) \otimes P_p(T)} \|\ell-q\|_{L_2(I_e \times T)}^2}\, \|\nabla_x v\|_{L_2(\Sigma_e)^d},
\end{split}
\ee
(cf.~\eqref{eq:g1}).

We estimate
\begin{align*}
r_e^\delta(v)=&r^\delta(\psi_e v)=\langle M \nabla_x (\theta_\delta-\widecheck{\theta}^\delta),\nabla_x (\psi_e v)\rangle_{L_2(\Sigma_e)^d}\\
\leq &\|\lambda_{\max}(M^{\frac12})\|_{L_\infty(\Sigma_e)}
\|M^{\frac12} \nabla_x(\theta_\delta-\widecheck{\theta}^\delta)\|_{L_2(\Sigma_e)^d}\\
& \times
(1+C_{PF} \diam(\omega_e) \|\nabla \phi_e\|_{L_\infty(\omega_e)})\|\nabla_x v\|_{L_2(\Sigma_e)^d}
\end{align*}
by \eqref{eq:52}. For a constant $C_2$
only dependent on the shape-regularity constant $\kappa_\bbP$, we have $C_{PF} \diam(\omega_e) \|\nabla \phi_e\|_{L_\infty(\omega_e)} \leq C_2$.
It follows that 
\be \label{eq:g9}
\sup_{0 \neq v \in P_p(I_e) \otimes H_*^1(\omega_e)} \frac{r_e^\delta(v)}{\|\nabla _x v\|_{L_2(\Sigma_e)^d}}
\leq 
 \|\lambda_{\max}(M^{\frac12})\|_{L_\infty(\Sigma_e)} (1+C_2) \|M^{\frac12} \nabla_x(\theta_\delta-\widecheck{\theta}^\delta)\|_{L_2(\Sigma_e)^d}.
\ee
The proof is concluded by combining \eqref{eq:g6},  \eqref{eq:g7}, \eqref{eq:g8}, and \eqref{eq:g9}. 
\end{proof}

\section{The double-adaptive loop} \label{sec:double-adaptive}
For $\delta \in \Delta\setminus\{\infty\}$, $X^\delta$ and $Y^\delta$ defined in \eqref{eq:105} and \eqref{eq:106}, and $(\lambda^\delta,u^\delta) \in Y^\delta \times X^\delta$ being the solution of \eqref{m9}, we set for $P \in \pria^\delta$
\begin{align*}
\eta_P(\delta)^2&:=\iint_P |M^{\frac12} \nabla_x \lambda^\delta|^2\diff t\diff x+\left\{\begin{array}{cl} \int_{T_P} |u_0-\gamma_0 u^\delta|^2\diff x& \text{ if } I_P \cap \{0\}\neq \emptyset,\\
0 & \text{ otherwise},\end{array}\right.
\intertext{and for $P \in \pria^{\underline{\delta}}$,}
\vartheta_P(\delta)^2&:=\iint_P |M^{-\frac12}\sigma^\delta+M^{\frac12} \nabla_x \widecheck{\theta}^\delta|^2 +|M^{\frac12} \nabla_x (\widecheck{\theta}^\delta-\theta^\delta)|^2 \diff t\diff x,
\end{align*}
where $\widecheck{\theta}^\delta$ is defined in \eqref{eq:102}, and $\sigma^\delta$ in Theorem~\ref{thm:reliability}.
Then  $\eta(\delta)$, defined in Theorem~\ref{thm:4}, satisfies
$\eta(\delta)^2=\sum_{P \in  \pria^\delta} \eta_P(\delta)^2$, and we set $\vartheta(\delta)^2:=\sum_{P \in  \pria^{\underline{\delta}}} \vartheta_P(\delta)^2$.

Using \eqref{eq:103} and \eqref{eq:104}, from Theorem~\ref{thm:reliability} it follows that $\|\lambda_\delta-\lambda^\delta\|_Y \leq \vartheta(\delta)$, modulo a higher order data-oscillation term. In our experiments, we observed that the latter term is always negligible, being on average anywhere between $8$-$12$ orders of magnitude smaller than $\vartheta(\delta)$ depending on the example, so that tacitly we ignore this term below. 

Fixing a constant $\varrho>0$, Theorem~\ref{thm:pjotr} shows then that
\emph{if}
\be \label{eq:pjotrmod}
\vartheta(\delta)^2 \leq \varrho^2(\|M^{\frac12} \nabla_x \lambda^\delta\|_{L_2(\Sigma)^d}^2+(1+\tfrac{1}{(\varrho+\frac12)^2}) \|u_0-\gamma_0 u^\delta\|_{L_2(\Omega)}^2),
\ee
\emph{then} $u^\delta$ is quasi-best in the sense that \eqref{eq:g12} is valid, and with
$$
\varrho_{\mathrm{eff}}(\delta):=\min \Big\{ \widehat\varrho \geq 0 \colon \eqref{eq:pjotrmod}\text{ holds with }\varrho \text{ replaced by }\widehat\varrho \Big\},\footnotemark
$$
\footnotetext{The right-hand of \eqref{eq:pjotrmod} is monotonically increasing as function of $\varrho$.}%
it holds that
\be \label{eq:g13}
\tfrac{1}{c(\alpha)} \eta(\delta)^2 \leq \|u-u^\delta\|_X^2 \leq c(\alpha)\big(1+\varrho_{\mathrm{eff}}(\delta)^2(1+\tfrac{1}{(\varrho_{\mathrm{eff}}(\delta)+\frac12)^2})\big) \eta(\delta)^2
\ee%
by Proposition~\ref{prop:1}.
These observations suggest the following adaptive solution method.\vspace*{0.6ex}

 {\rm%
\begin{algotab}
\> \bf{do}\=\bf{uble-adaptive solver}$(\ell,u_0, \rm{TOL})$\\
\>\> Let $\delta=\underline{\delta}=\delta_0$. Select constants $\varrho>0$ and $\varpi \in (0,1)$.\\
\>\> (I) Solve $(\lambda^\delta,u^\delta) \in Y^\delta \times X^\delta$ from \eqref{m9}.\\
\>\> \texttt{if} $\vartheta(\delta)$ does not satisfy \eqref{eq:pjotrmod}\\
\>\> \texttt{then} \= select the smallest ${\mathcal M} \subset \pria^{\underline{\delta}}$ such that $\sum_{P \in{\mathcal M}} \vartheta_P(\delta)^2 \geq \varpi \vartheta(\delta)^2$, and\\
\>\>\>  determine the smallest $\underline{\underline{\delta}} \succeq \underline{\delta}$ such that $\pria^{\underline{\underline{\delta}}} \cap {\mathcal M} =\emptyset$ (use {\bf refine}).\\
\>\>\>  Set $\underline{\delta}:=\underline{\underline{\delta}}$, and \texttt{goto} (I).\\
\>\> \texttt{elseif} $c(\alpha)\big(1+\varrho_{\mathrm{eff}}(\delta)^2(1+\tfrac{1}{(\varrho_{\mathrm{eff}}(\delta)+\frac12)^2})\big) \eta(\delta)^2 \leq \rm{TOL}^2$ \texttt{then STOP} \\
\>\>\texttt{else} \= select the smallest ${\mathcal M} \subset \pria^{\delta}$ such that $\sum_{P \in{\mathcal M}} \eta_P(\delta)^2 \geq \varpi \eta(\delta)^2$, and\\
\>\>\> determine the smallest $\tilde{\delta} \succeq \delta$ such that $\pria^{\tilde{\delta}} \cap {\mathcal M} =\emptyset$ (use {\bf refine}).\\
\>\>\> Set $\delta=\underline{\delta}:=\tilde{\delta}$, and \texttt{goto} (I).\\
\>\>  \texttt{endif} 
\end{algotab}}

\begin{remark}\label{rem:alternative-rho}
Instead of letting the number of inner test-space enrichments depend on validity of \eqref{eq:pjotrmod}, alternatively one can fix this number in advance, possibly equal to zero.
A disadvantage is that then quasi-optimality --- in the sense that for some fixed $\rho$ the produced sequence of $u^\delta$ satisfy \eqref{eq:g12} --- is not guaranteed a priori.
\end{remark}

\begin{remark}\label{rem:pjotr-benadering}
The construction of the estimator $\vartheta(\delta)$ requires the auxiliary time-slab partition $\widecheck{\pria}^{\underline{\delta}}$ and the corresponding equilibrated flux reconstructions. As observed in Remark~\ref{rem:slabbifiedY}, the size of $\widecheck{\pria}^{\underline{\delta}}$ cannot in general be bounded uniformly in terms of the size of $\pria^{\underline{\delta}}$. We therefore also consider the following simpler, but heuristic, alternative.

Let
\[
    Z^\delta \coloneqq L_2(\Xi;H_0^1(\Omega)) \cap Q_{p+1,p+1}(\pria^{\underline{\delta}}),
\]
i.e., the analogue of $Y^\delta$ on the same prismatic partition, but with polynomial degrees increased by one in both time and space. Let $\theta_Z^\delta\in Z^\delta$ be defined by
\[
    \langle M\nabla_x \theta_Z^\delta,\nabla_x v\rangle_{L_2(\Sigma)^d} = (\ell-B_a u^\delta)(v) \qquad (v\in Z^\delta).
\]
Since $Y^\delta\subset Z^\delta$, Galerkin orthogonality gives the exact
Pythagoras identity
\[
    \|\theta_\delta-\theta^\delta\|_Y^2 = \|\theta_\delta-\theta_Z^\delta\|_Y^2 + \|\theta_Z^\delta-\theta^\delta\|_Y^2 .
\]
The computable quantity
\[
    \vartheta_Z(\delta) \coloneqq \|\theta_Z^\delta-\theta^\delta\|_Y
\]
may therefore be used as a surrogate for $\|\theta_\delta-\theta^\delta\|_Y = \|\lambda_\delta-\lambda^\delta\|_Y$, provided that the higher-order Galerkin error $\|\theta_\delta-\theta_Z^\delta\|_Y$ is negligible.

This replacement avoids both the construction of $\widecheck{\pria}^{\underline{\delta}}$ and the solution of the local flux equilibration problems. On the other hand, it does not provide a guaranteed upper bound for $\|\theta_\delta-\theta^\delta\|_Y$, and the resulting criterion should therefore be regarded as empirical. Moreover, the space $Z^\delta$ can contain substantially more degrees of freedom than $Y^\delta$, and the associated problem does not in general decompose into independent slabwise elliptic solves. The numerical experiments below investigate whether this higher-order surrogate nevertheless provides a useful practical alternative to the certified estimator $\vartheta(\delta)$.
\end{remark}

\section{Numerical results}
\label{sec:numerics}

In this section we present numerical experiments for the heat equation on the space-time cylinder $\Sigma=\Xi\times\Omega$. 
That is, we consider \eqref{eq:33} with $A=A_s$, so that $c(\alpha)=1$, and $(A_s w)(v)=\iint_\Sigma \nabla_x w \cdot \nabla_x v\diff t \diff x$. The discrete solution is computed from the saddle-point formulation \eqref{m9}. The purpose of the experiments is to compare uniform and adaptive refinement, and to investigate the behaviour of the test-space refinement mechanism for solutions with limited regularity.

The trial spaces $X^\delta$ and test spaces $Y^\delta$ are those from \eqref{eq:105} and \eqref{eq:106}. The refinement procedure is the one described in Section~\ref{sec:apost}. In dimensions $d>1$, the spatial refinements are generated by newest-vertex bisection. In dimension $d=1$, we impose the additional spatial grading condition from Remark~\ref{rem:one-dimensional-grading}. For the adaptive computations we use the double-adaptive loop from Section~\ref{sec:double-adaptive} with constants $\varrho=1$ and $\varpi = 0.4$.
We compare the following two refinement strategies. The runs labelled by \texttt{unif} use uniform refinement of the trial partition, and take an equal test partition. Example~\ref{ex:tensor} shows that in this case, the inf-sup stability \eqref{eq:infsup2} holds true. The runs labelled by \texttt{adap} use the double-adaptive loop from Section~\ref{sec:double-adaptive}.
To be able to compare both strategies we used also for the uniform refinement case the upper bound for $\|u-u^\delta\|_X$ from \eqref{eq:g13}, which requires the additional computation of $\vartheta(\delta)$.

We also tested the variant suggested in Remark~\ref{rem:alternative-rho}, in which no inner enrichments of the test space are performed. 
Since the results for this variant closely matched those for the standard adaptive runs, we omit them from the figures.

Finally, in Section~\ref{sec:numerics-estimator-diagnostics}, we monitor the higher-order surrogate from Remark~\ref{rem:pjotr-benadering}. This surrogate replaces the certified test-space estimator $\vartheta(\delta)$ --- which is based on the space corresponding to the auxiliary time-slab mesh $\widecheck\pria^{\underline\delta}$ and the local flux equilibration problems --- by the computable quantity $\vartheta_Z(\delta)$ on the higher-order space $Z^\delta$. We compare $\vartheta_Z(\delta)$ with the guaranteed upper bound $\vartheta(\delta)$. We also inspect the ratio between the two parts of $\vartheta(\delta)$: the slab error term $\|\theta^\delta - \widecheck{\theta}^\delta\|_Y$ and the flux error term $\|M^{-\frac{1}{2}}\sigma^\delta + M^{\frac{1}{2}}\nabla_x \widecheck{\theta}^\delta\|_{L_2(\Sigma)^d}$.

Throughout the experiments, we take
\[
    \Xi=(0,1),\qquad
    \Omega=(0,1)^d,\qquad
    d\in\{1,2\}.
\]
For $d=1$, the initial space-time mesh is obtained by subdividing the square $(0,1)^2$ into four congruent squares by splitting both coordinate directions at their midpoint.
For $d=2$, the initial spatial mesh consists of four triangles obtained by connecting the centre point $(1/2,1/2)$ to the four corners of the unit square. In this initial mesh, the longest edges are chosen as refinement edges for newest vertex bisection. The initial space-time mesh is then obtained by tensorizing this spatial mesh with the time interval $\Xi$, resulting in four triangular prisms. 

The implementation is written in \Cpp{} and is available online.\footnote{Repository link: \url{https://github.com/Rsmeets99/DoubleAdapParabolicFEM}.} 
The repository contains implementation details, input files, and the data used to generate the plots reported below. Numerical results are shown for polynomial
degrees
\[
    p=1,2,3,
\]
although the implementation is not restricted to these values.

As a consistency check, we also performed standard manufactured-solution tests with smooth data under both uniform and adaptive refinement, for $d=1$ and $d=2$ and for all polynomial degrees considered here. The observed convergence rates agree with the expected optimal rates. Since these tests primarily serve to validate the implementation, they are not included below. Instead, we focus on examples in which the solution has reduced regularity due either to an incompatibility between the initial and boundary data, or to reduced regularity of the initial datum.
The examples are based on the benchmark problems considered in~\cite{75.291}, where a least-squares method is investigated. 

\subsection{Experiments in \(1+1\) dimensions}
\label{subsec:experiments-1d}

We first consider the case $d=1$, so that $\Omega=(0,1)$. Homogeneous Dirichlet boundary conditions are imposed at the endpoints of the interval.

\subsubsection{Non-matching initial datum}
\label{subsubsec:1d-incompatible-initial-datum}

The first one-dimensional experiment uses a non-matching initial datum. We take 
\[
    \ell = 2, \qquad
    u_0(x) = 1 .
\]
The initial datum is admissible as an element of $L_2(\Omega)$, but it is incompatible with the homogeneous Dirichlet boundary condition at $x=0$ and $x=1$. Consequently, the solution is discontinuous at the points $(t,x)=(0,0)$ and $(t,x)=(0,1)$.

Figure~\ref{fig:non-matching-conv-1D} shows the convergence rates of the guaranteed a posteriori upper bound $\sqrt{\big(1+\varrho_{\mathrm{eff}}(\delta)^2(1+\tfrac{1}{(\varrho_{\mathrm{eff}}(\delta)+\frac12)^2})\big)}\,\eta(\delta)$ for $\|u-u^\delta\|_X$ from \eqref{eq:g13}
 against the degrees of freedom of $X^\delta$. The rates appear to be independent of the polynomial degree $p$: they are $0.12$ for the uniform refinement and approximately $0.5$ for the adaptive refinement. Note that for a smooth solution, the rate would be $p/2$. A snapshot of the mesh for adaptive refinement is given in Figure~\ref{fig:non-matching-mesh-1D}, where clear refinement is seen towards the singular corners at $(t,x) = (0,0)$ and $(t,x) = (0,1)$. These results match the findings in \cite{75.291}, with a slightly increased rate for the adaptive $p=1$ case, and adding results for $p=2,3$. 
 
For the adaptive algorithm we find for $\varrho_{\rm eff}(\delta)$ passing the test \eqref{eq:pjotrmod} an average value of $0.70$, $0.57$\, and $0.55$ for $p=1,2,3$, respectively.
On average, the amount of inner test space enrichments required per outer trial space enrichment was $0.1$, $0$, and $0.1$ for $p=1,2,3$\, respectively, and the ratios $\#Y^\delta / \#X^\delta$ were found to be $2.00$, $1.47$\, and $1.33$ for $p=1,2,3$\, respectively. These ratios are close to $(p+1)/p$, 
being the asymptotic value of these quantities when $Y^\delta$ and $X^\delta$ are the corresponding finite element spaces w.r.t.~the same uniform prismatic partition.

For the variant discussed in Remark~\ref{rem:alternative-rho} without inner test-space enrichments, the average value of $\varrho_{\rm eff}(\delta)$ was 
$0.75$, $0.57$\, and $0.65$ for $p=1,2,3$\, respectively, with maximum value $1.25$.

\subsubsection{Non-smooth initial datum}
\label{subsubsec:1d-boundary-singularity}

The second one-dimensional experiment uses an initial datum with reduced regularity. We take
\[
    \ell = 0, \qquad
    u_0(x) = x^{1/2}(1-x).
\]
This initial datum satisfies the homogeneous Dirichlet boundary condition. Its derivative is singular at $x=0$, which reduces the regularity of the corresponding solution. Thus, unlike in the previous example, the loss of regularity is not caused by a mismatch between the initial and boundary data, but by the non-smoothness of 
 the initial datum itself.

Figure~\ref{fig:reduced-boundary-conv-1D} shows the convergence rates of the a posteriori upper bound against the degrees of freedom of $X^\delta$. The rates are $0.26$ for the uniform refinement and $0.5$, $0.88$ and $0.93$ for the adaptive refinement for $p =1,2,3$ respectively. A snapshot of the mesh for adaptive refinement is given in Figure~\ref{fig:reduced-boundary-mesh-1D}, where clear refinement is seen towards the reduced-regularity corner at $(t,x) = (0,0)$. These results match the findings in \cite{75.291} with the addition of results for $p=2,3$.

For the adaptive algorithm we find for $\varrho_{\rm eff}(\delta)$ passing the test \eqref{eq:pjotrmod} an average value of $0.67$, $0.64$\, and $0.67$ for $p=1,2,3$\, respectively. 
On average\, the amount of inner iterations required per outer iteration was $0.6$, $0.2$\, and $0.3$ for $p=1,2,3$\, respectively, and the ratios $\#Y^\delta / \#X^\delta$ were found to be $2.5$, $1.57$ and $1.47$ for $p=1,2,3$ respectively.

For the variant discussed in Remark~\ref{rem:alternative-rho} without inner test-space enrichments, the average value of $\varrho_{\rm eff}(\delta)$ was 
$1.26$, $0.91$\, and $1.02$ for $p=1,2,3$\, respectively, with maximum value $3.14$.

\begin{figure}[htbp]
    \centering
    \begin{subfigure}[t]{0.49\textwidth}
        \centering
        \includegraphics[width=\linewidth]{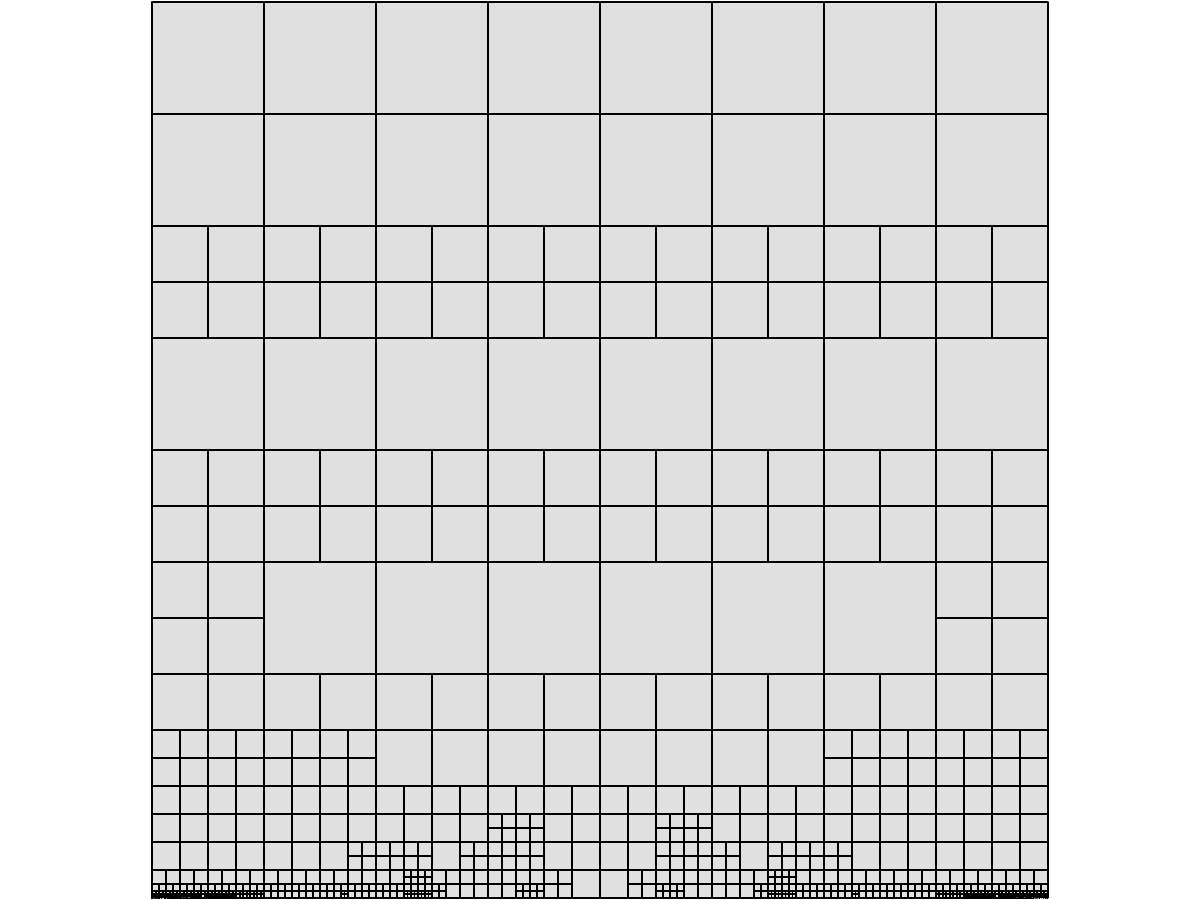}
        \caption{Non-matching initial datum}
        \label{fig:non-matching-mesh-1D}
    \end{subfigure}
    \hfill
    \begin{subfigure}[t]{0.49\textwidth}
        \centering
        \includegraphics[width=\linewidth]{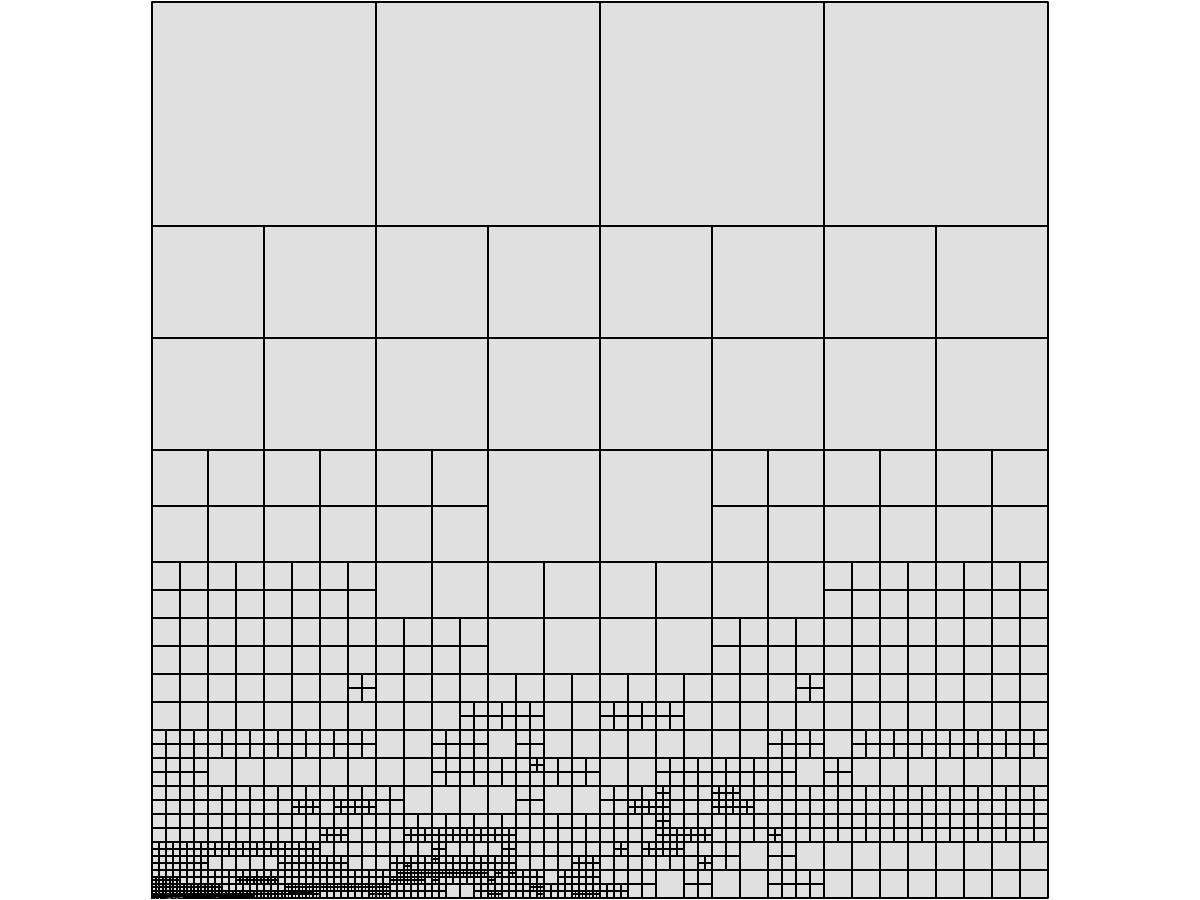}
        \caption{Non-smooth initial datum}
        \label{fig:reduced-boundary-mesh-1D}
    \end{subfigure}
    \caption{Snapshots of the meshes belonging to $X^\delta$ with approximately 2000 degrees of freedom for $d = 1$ and $p = 1$.}
    \label{fig:meshes-1D}
\end{figure}

\begin{figure}[htbp]
    \centering
    \begin{subfigure}[t]{0.8\textwidth}
        \centering
        \includegraphics[width=\linewidth]{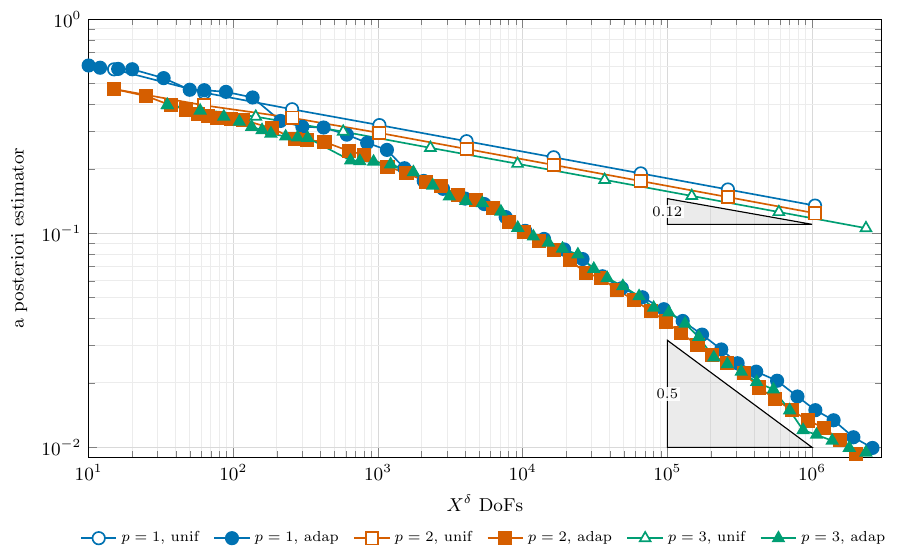}
        \caption{Non-matching initial datum}
        \label{fig:non-matching-conv-1D}
    \end{subfigure}
    \vfill
    \begin{subfigure}[t]{0.8\textwidth}
        \centering
        \includegraphics[width=\linewidth]{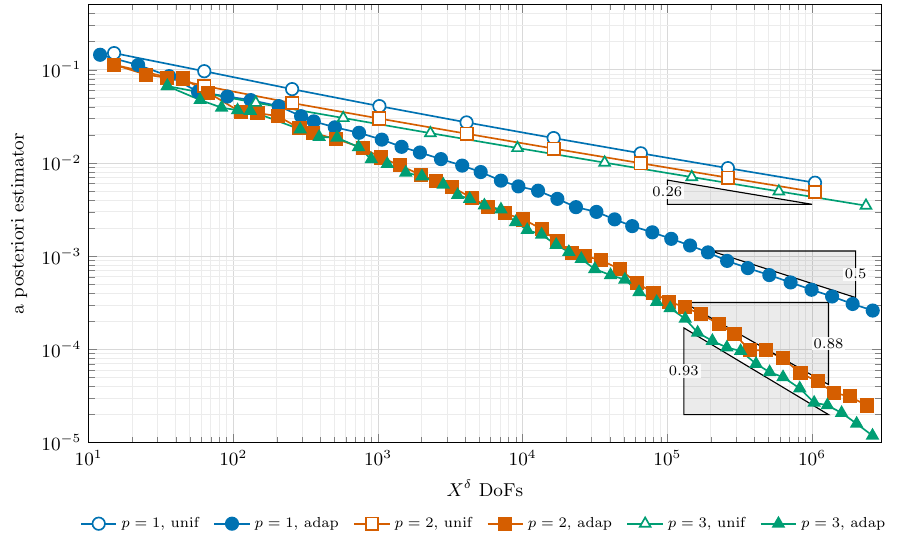}
        \caption{Non-smooth initial datum}
        \label{fig:reduced-boundary-conv-1D}
    \end{subfigure}
    \caption{Convergence rates for the $d = 1$ examples.}
    \label{fig:conv-rates-1D}
\end{figure}

\subsection{Experiments in \(2+1\) dimensions}
\label{subsec:experiments-2d}

We next consider the $d=2$ case, with $\Omega=(0,1)^2$. Homogeneous Dirichlet boundary conditions are imposed on $\partial\Omega$.

\subsubsection{Non-matching initial datum}
\label{subsubsec:2d-incompatible-initial-datum}

The two-dimensional experiment with non-matching initial datum is given by
\[
    \ell = 0, \qquad
    u_0(x,y) = 1 .
\]
As in the one-dimensional case, the initial datum does not match the homogeneous Dirichlet boundary condition. The incompatibility is now present along the full spatial boundary at $t=0$. This produces reduced regularity near the parabolic boundary $\{0\}\times\partial\Omega$, and in particular near the spatial corners.

Figure~\ref{fig:non-matching-conv-2D} shows the convergence rates of the  a posteriori upper bound  against the degrees of freedom of $X^\delta$. The rates are $0.08$ for the uniform refinement and $0.14$ for the adaptive refinement. Note that for a smooth solution, the rate would be $p/3$. A snapshot of the mesh for adaptive refinement is given in Figure~\ref{fig:non-matching-mesh-2D}, where clear refinement is seen towards the singularity at $\{0\} \times \partial \Omega$. These results match the findings in \cite{75.291}, with the addition of results for $p=2,3$.

For the adaptive algorithm we find an average value for $\varrho_{\rm eff}(\delta)$ of $0.075$, $0.071$\, and $0.069$ for $p=1,2,3$\, respectively. All outer iterations required no inner iteration\, and the ratios $\#Y^\delta / \#X^\delta$ were found to be $1.91$, $1.45$\, and $1.29$ for $p=1,2,3$\, respectively.

\subsubsection{Non-smooth initial datum}
\label{subsubsec:2d-boundary-singularity}

Finally, we consider a two-dimensional initial datum with reduced regularity at a boundary edge:
\[
    \ell = 0, \qquad
    u_0(x,y) = x^{3/4}(1-x)y(1-y).
\]
The datum vanishes on $\partial\Omega$, and is therefore compatible with the homogeneous Dirichlet boundary condition. Its spatial regularity is nevertheless reduced near the boundary edge $x=0$, where the derivative in the $x$-direction is singular.

Figure~\ref{fig:reduced-boundary-conv-2D} shows the convergence rates of the  a posteriori upper bound  against the degrees of freedom of $X^\delta$. The rates are $0.27$ and approximately $0.33$ for the uniform refinement for $p=1$ and $p=2,3$ respectively and $0.33$ and $0.46$ for the adaptive refinement for $p=1$ and $p=2,3$ respectively. A snapshot of the mesh for adaptive refinement is given in Figure~\ref{fig:reduced-boundary-mesh-2D}, where clear refinement is seen towards the singularity at $\{0\} \times \{x = 0\}$. These results match the findings in \cite{75.291}, with the addition of results for $p=2,3$. 

For the adaptive algorithm we find an average value for $\varrho_{\rm eff}(\delta)$ of $0.42$, $0.22$\, and $0.22$ for $p=1,2,3$\, respectively. All outer iterations required no inner iteration\, and the ratios $\#Y^\delta / \#X^\delta$ were found to be $1.98$, $1.45$\, and $1.29$ for $p=1,2,3$\, respectively.

\begin{figure}[htbp]
    \centering
    \begin{subfigure}[t]{0.49\textwidth}
        \centering
        \includegraphics[width=\linewidth]{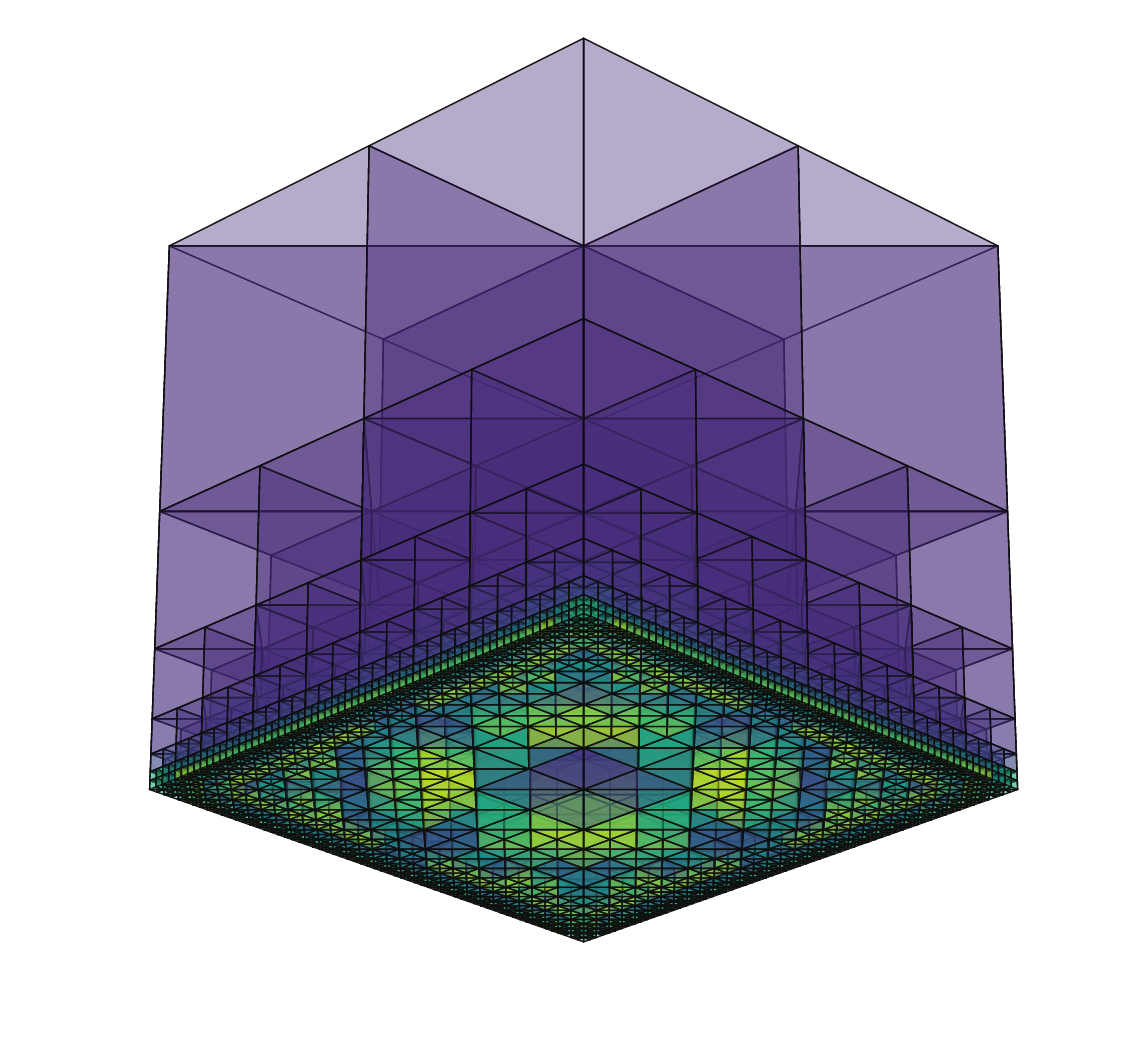}
        \caption{Non-matching initial datum}
        \label{fig:non-matching-mesh-2D}
    \end{subfigure}
    \hfill
    \begin{subfigure}[t]{0.49\textwidth}
        \centering
        \includegraphics[width=\linewidth]{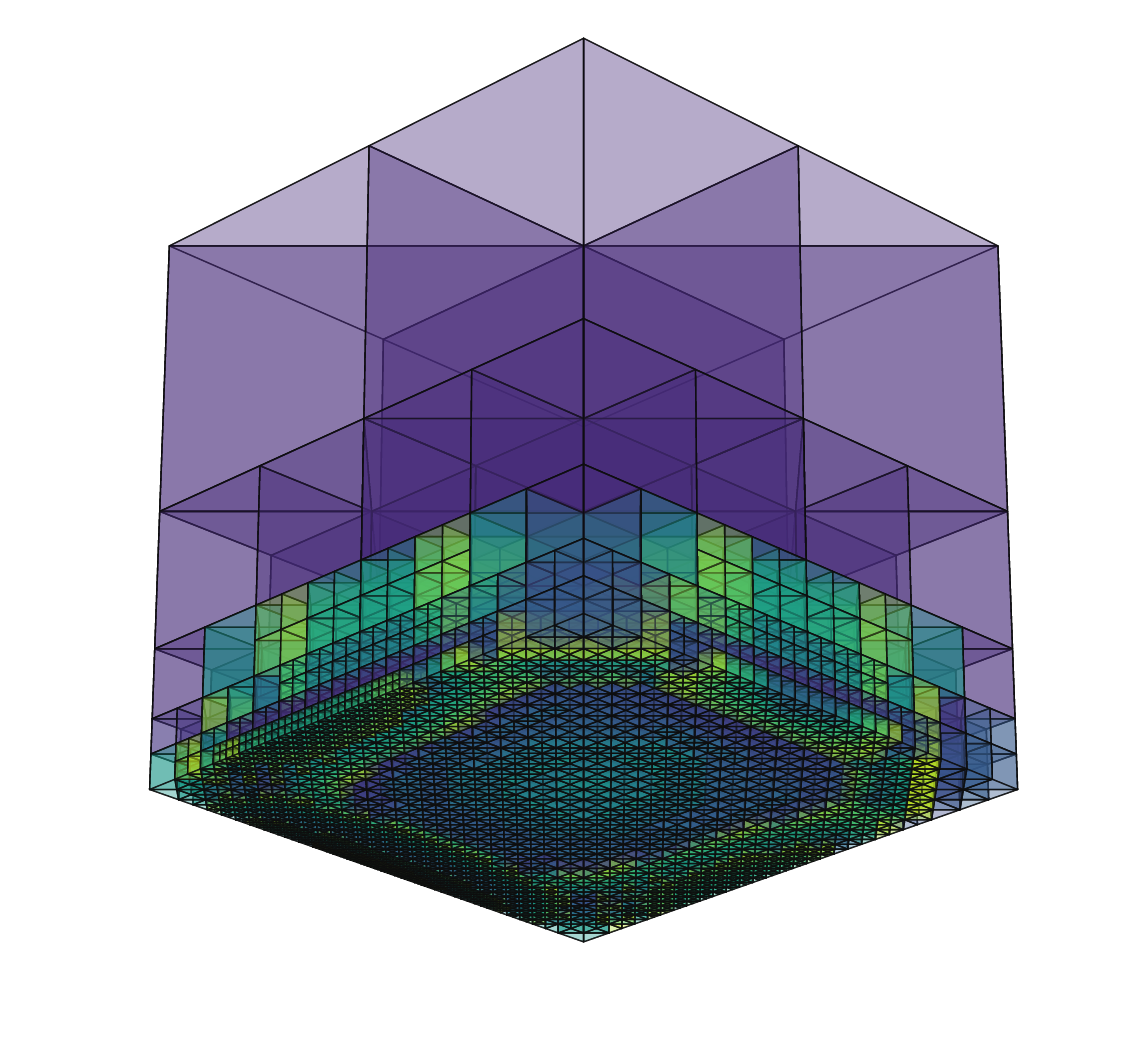}
        \caption{Non-smooth initial datum}
        \label{fig:reduced-boundary-mesh-2D}
    \end{subfigure}
    \caption{Snapshots of the meshes belonging to $X^\delta$ with approximately 4000 degrees of freedom for $d = 2$ and $p = 1$.}
    \label{fig:meshes-2D}
\end{figure}

\begin{figure}[htbp]
    \centering
    \begin{subfigure}[t]{0.8\textwidth}
        \centering
        \includegraphics[width=\linewidth]{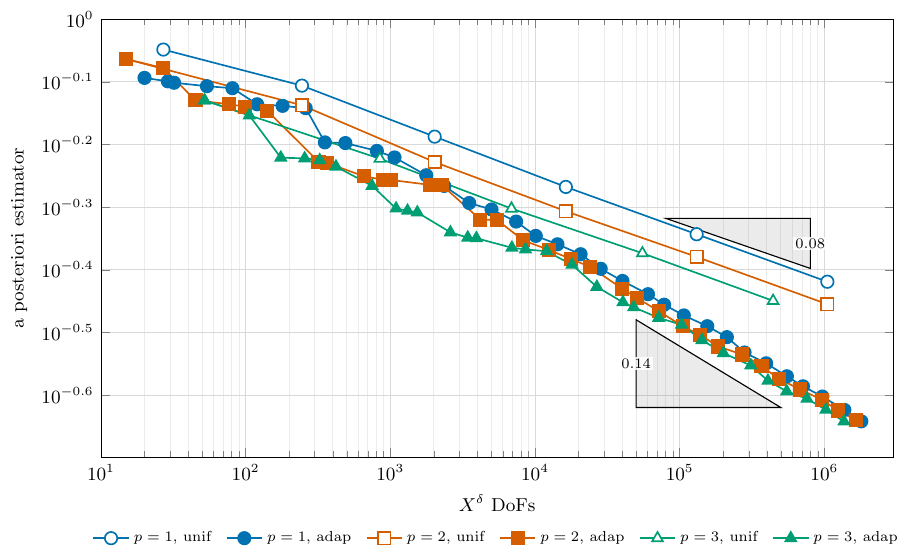}
        \caption{Non-matching initial datum}
        \label{fig:non-matching-conv-2D}
    \end{subfigure}
    \vfill
    \begin{subfigure}[t]{0.8\textwidth}
        \centering
        \includegraphics[width=\linewidth]{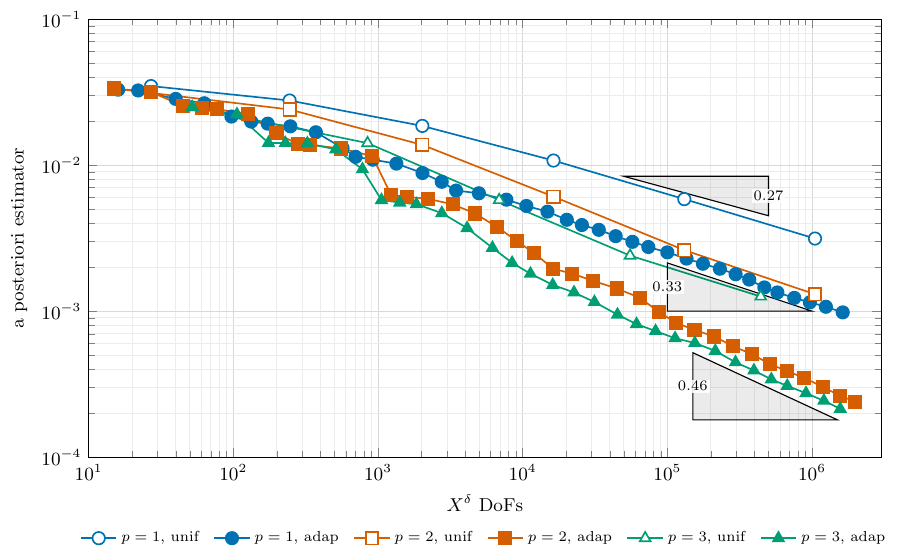}
        \caption{Non-smooth initial datum}
        \label{fig:reduced-boundary-conv-2D}
    \end{subfigure}
    \caption{Convergence rates for the $d = 2$ examples.}
    \label{fig:conv-rates-2D}
\end{figure}

\subsection{Diagnostics for the test-space estimator}
\label{sec:numerics-estimator-diagnostics}

We conclude the numerical section by comparing the certified test-space estimator $\vartheta(\delta)$ for $\|\theta_\delta-\theta^\delta\|_Y$ with the higher-order surrogate $\vartheta_Z(\delta)=\|\theta_Z^\delta-\theta^\delta\|_Y$ from Remark~\ref{rem:pjotr-benadering}. Recall that $\vartheta(\delta)$ consists of two parts: the contribution from the local flux equilibration problems, and the computable Galerkin difference caused by replacing $Y^\delta$ by the space corresponding to the auxiliary time-slab mesh $\widecheck\pria^{\underline\delta}$. We write
\[
    \vartheta_{\rm flux}(\delta) \coloneqq \|M^{-\frac12}\sigma^\delta + M^{\frac12}\nabla_x\widecheck\theta^\delta\|_{L_2(\Sigma)^d}, \qquad 
    \vartheta_{\rm slab}(\delta) \coloneqq \|\widecheck\theta^\delta-\theta^\delta\|_Y .
\]
Modulo the data-oscillation term that is ignored in the computations, this gives
\[
    \vartheta(\delta)^2 = \vartheta_{\rm flux}(\delta)^2 + \vartheta_{\rm slab}(\delta)^2 .
\]

In the plots below we report
\[
    r_Z(\delta) \coloneqq
    \frac{\vartheta_Z(\delta)}{\vartheta(\delta)}\quad\text{and}\quad
    r_{\vartheta}(\delta) \coloneqq
    \frac{\vartheta_{\rm slab}(\delta)}{\vartheta_{\rm flux}(\delta)} .
\]
The ratio $r_Z(\delta)$ measures how close the higher-order surrogate is to the certified upper bound. Since $\vartheta_Z(\delta)$ is not guaranteed to be an upper bound for $\|\lambda_\delta-\lambda^\delta\|_Y$, this ratio is used only as an empirical diagnostic. The ratio $r_{\vartheta}(\delta)$ indicates which part of the estimator $\vartheta(\delta)$ is dominant. 

\begin{figure}[htbp]
    \centering
    \begin{subfigure}[t]{0.49\textwidth}
        \centering
        \includegraphics[width=\linewidth]{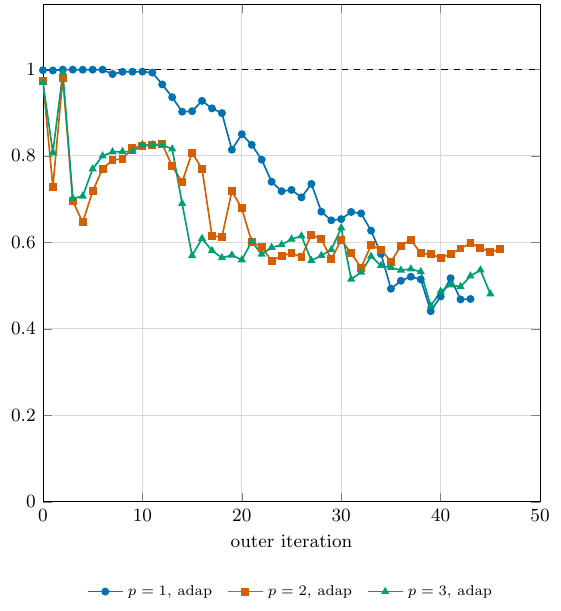}
        \caption{Ratio $r_Z(\delta)$}
        \label{fig:ratio-Z-1D}
    \end{subfigure}
    \hfill
    \begin{subfigure}[t]{0.49\textwidth}
        \centering
        \includegraphics[width=\linewidth]{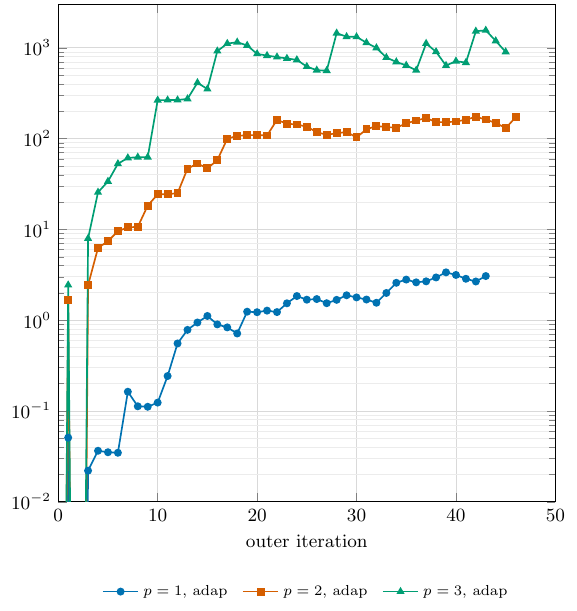}
        \caption{Ratio $r_\vartheta(\delta)$}
        \label{fig:ratio-vartheta-1D}
    \end{subfigure}
    \caption{Ratios $r_Z(\delta)$ and $r_\vartheta(\delta)$ for the non-matching initial datum for $d = 1$.}
    \label{fig:diag-ratios}
\end{figure}

As the corresponding plots for the other examples and polynomial degrees showed the same qualitative behaviour, we only display Figure~\ref{fig:diag-ratios}. Figure~\ref{fig:ratio-Z-1D} shows that $\vartheta_Z(\delta)$, although not a reliable estimator for $\vartheta(\delta)$, remains of the same order as the certified estimator. This supports its use as an empirical, and potentially cheaper, diagnostic. Figure~\ref{fig:ratio-vartheta-1D} shows that, in these runs, $\vartheta(\delta)$ is mostly dominated by the slab reconstruction error $\vartheta_{\rm slab}(\delta)$, rather than by the flux-equilibration upper bound $\vartheta_{\rm flux}(\delta)$ for $\|\theta_\delta-\widecheck{\theta}^\delta\|_Y$. This indicates that the computed estimator is sharp in these experiments. The sharp drops in the first outer iterations are caused by $Y^\delta=\widecheck{Y}^\delta$, for which $\vartheta_{\rm slab}(\delta)=0$.

\section{Conclusion} \label{sec:conclusion}

We studied minimal residual space-time finite element discretizations of linear parabolic initial value problems.
To deal with the arising dual norm, we introduced the Riesz lift of the residual as a secondary variable.
Quasi-optimality of the primal variable of the mixed system follows from a uniform inf-sup condition.
Whereas this condition is known to be satisfied for finite element spaces w.r.t.~prismatic partitions of the space-time cylinder that allow for a decomposition into time-slabs, we proved
that this condition cannot be expected to hold for non-time-slab partitions.
In view of this obstruction, we derived a data-dependent a posteriori condition on the test space that guarantees quasi-optimality for the data at hand. Moreover, under this condition, an a posteriori estimator for the error in the principal variable is both efficient and reliable, with explicit constants.

In order to be able to verify the aforementioned a posteriori condition, we proposed a computable estimator for the error in the secondary variable.
The construction is based on the coarsest time-slab refinement of the partition and equilibrated flux reconstructions, and is efficient and reliable modulo higher order data oscillation. Together with the a posteriori estimator for the error in the principal variable, this leads to a double-adaptive loop: in the inner loop the test space is enlarged until the data-dependent stability condition is satisfied, and in the outer loop the trial space is enlarged using the principal error indicators. The numerical experiments for singular heat-equation benchmarks show that the method refines towards the expected singularities and improves the rates compared with uniform refinement.

As already announced in the introduction, the extension to parabolic problems with a non-linear spatial operator that is Lipschitz continuous and strongly monotone will be the topic of a forthcoming work that builds on the combination of \cite{19.98} and the current work. Besides this extension, two issues deserve particular attention. The first is the construction and analysis of an optimal preconditioner for the saddle-point systems arising in the locally refined space-time setting. The second is the development of a different estimator for the inner loop that does not require the construction of the auxiliary time-slab partition. Such an estimator would remove the main practical drawback of the certified test-space estimator used here.

\appendix
\section{The necessity of time-slabs for conditions \eqref{eq:infsup2} and \eqref{eq:infsup1}} \label{sec:necessity}
Let $\Xi=(0,1)$, and for some bounded Lipschitz domain $\Omega \subset \R^d$, let $(H,V)=(L_2(\Omega),H^1_0(\Omega))$.
For our goal, instead of with $\sqrt{(A_s\bigcdot)(\bigcdot)}$, here we may equip $Y=L_2(\Xi;H^1_0(\Omega))$ with its canonical norm $\|\nabla_x\bigcdot\|_{L_2(\Sigma)^d}$.

Let $X^\delta$ and $Y^\delta$ be \emph{finite element spaces} w.r.t.~a joint, shape-regular, isotropic prismatic partition of $\Sigma$.
For $d \in \{1,2\}$, we will demonstrate that when such partitions do not allow a subdivision into time slabs, in general, condition~\eqref{eq:infsup2}, and even condition~\eqref{eq:infsup1} \emph{cannot} be expected to be valid.
Our results extend to the situation where the partition that underlies $Y^\delta$ is a refinement of some fixed maximal depth of the partition that underlies $X^\delta$.
Our arguments will be based on the following alternative characterization of the inf-sup constant $\gamma_\delta^{\diff_t}$ defined in \eqref{eq:56}.
As shown in \cite[Prop.~2.5]{35.8565}, it holds that
\be \label{eq:59}
\sup_{\{w\in X^\delta\colon \diff_t w \neq 0\}} \inf_{v \in Y^\delta} \frac{\|R_Y^{-1}\diff_t w  -v\|_Y}{\|R_Y^{-1}\diff_t w\|_Y}=\sqrt{1-(\gamma_\delta^{\diff_t})^2},
\ee
with the Riesz map $R_Y$ defined as in Section~\ref{sec:notations}.
In the literature on DPG methods (e.g.~\cite{64.14}), $R_Y^{-1}\diff_t w$ is known as the optimal test function for the trial function $w \in X^\delta$  (and operator $\diff_t\colon X \rightarrow Y'$). 
Actually, we will use the following stronger statement.
\begin{lemma}
For \emph{each} $w\in X^\delta$ with $\diff_t  w \neq 0$, it holds that
\be \label{eq:92}
\inf_{v \in Y^\delta} \frac{\|R_Y^{-1}\diff_t  w  -v\|^2_Y}{\|R_Y^{-1}\diff_t  w\|^2_Y}=1-\frac{\|{E_Y^\delta}'\diff_t  w\|_{{Y^\delta}'}^2}{\|\diff_t  w\|_{Y'}^2}.
\ee
\end{lemma}
\begin{proof}
Let  $K^\delta$ denote the $Y$-orthogonal projector onto $Y^\delta$. Then, we have
\begin{align*}
&\inf_{v \in Y^\delta} \frac{\|R_Y^{-1}\diff_t  w  -v\|^2_Y}{\|R_Y^{-1}\diff_t  w\|^2_Y}=
\frac{\|(\identity -K^\delta)R_Y^{-1}\diff_t  w \|^2_Y}{\|R_Y^{-1}\diff_t  w\|^2_Y}=
\frac{\|R_Y^{-1}\diff_t  w \|^2_Y-\|K^\delta R_Y^{-1}\diff_t  w\|^2_Y}{\|R_Y^{-1}\diff_t  w\|^2_Y}\\
&=1-\sup_{0 \neq v \in Y^\delta} \frac{\langle R_Y^{-1}\diff_t  w,v\rangle_Y^2}{\|v\|^2_Y \|\diff_t  w\|^2_{Y'}}=
1-\sup_{0 \neq v \in Y^\delta} \frac{(\diff_t  w)(v)^2}{\|v\|^2_Y \|\diff_t  w\|^2_{Y'}}=
1-\frac{\|{E_Y^\delta}'\diff_t  w\|^2_{{Y^\delta}'}}{\|\diff_t  w\|^2_{Y'}}.\,\,\Box\mbox{$\qedhere$}
\end{align*}
\end{proof}

\subsection{The case $d=1$}
In this case, $\Omega$ is an open interval. 
For some $p \in \N_0$, and $h_\delta, H_\delta >0$ with $p+2 \leq \frac{H_\delta}{h_\delta} \in \N$, let $\pria^\delta$ be a partition of $\Sigma$ into squares with edge lengths ranging from $h_\delta$ to $H_\delta$, the relevant part of which is illustrated in Figure~\ref{fig:1}.
\begin{figure}[h]
\centering
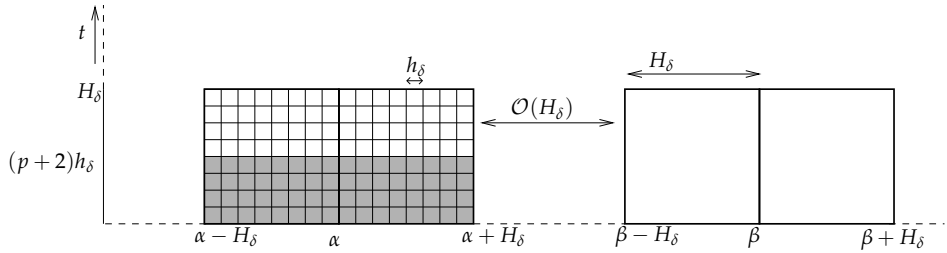
\caption{Partition of $\Xi \times \Omega$ into squares with $p+2\le \frac{H_\delta}{h_\delta}\in\N$.}
\label{fig:1}
\end{figure}
Let
$$
X^\delta := C(\overline{\Sigma}) \cap L_2(\Xi;H_0^1(\Omega)) \cap Q_{1,1}(\pria^\delta),
$$
and, for some $q \in \N$, let 
$$
Y^\delta:=L_2(\Xi;H_0^1(\Omega)) \cap Q_{p,q}(\pria^{\delta}).
$$

We will crucially exploit the following elementary result. For completeness, we provide the simple proof.

\begin{lemma}
There exists a non-zero sequence $(a_i)_{1 \leq i \leq p+2}$, with $\sum_{i=1}^{p+2} a_i=0$, such that the piecewise constant function $g^\delta$ defined by $g^\delta|_{((i-1) h_\delta,i h_\delta)}=a_i$ satisfies
$$
\int_0^{(p+2)h_\delta} g^\delta(s) k(s)\diff s=0 \quad (k \in P_p(0,(p+2)h_\delta)).
$$
\end{lemma}

\begin{proof}
It suffices to construct $g^\delta$ for $h_\delta=1$.
Let $Z_p$ be the space of piecewise constants w.r.t.~the partition of $[0,p+2]$ into $p+2$ subintervals of length $1$, and let $V_p$ be its subspace of functions in $Z_p$ that are constant on $(0,2)$. If there exists a $0 \neq g^\delta \in V_p$ that is orthogonal to 
$P_p(0,p+2)$, then we are done.
If such a function does not exist, then from $\dim V_p=\dim P_{p}(0,p+2)$ it follows that there exists a projector $\Pi_p$ defined on $L_2(0,p+2)$ with $\ran \Pi_p=V_p$ and $\ran (\identity - \Pi_p)=P_p(0,p+2)^\perp$.
Now let $z|_{(0,1)}=1$, and $z|_{(1,p+2)}=0$. From $z \in Z_p \setminus V_p$, $g^\delta:=(\identity - \Pi_p) z \in Z_p$ is non-zero, and orthogonal to $P_p(0,p+2)$. 
The property $\sum_{i=1}^{p+2} a_i=0$ follows from $g^\delta \perp 1$.
\end{proof}

In the following, we extend $g^\delta$ with zero outside $(0,(p+2)h_\delta)$.
Thanks to $\sum_{i=1}^{p+2} a_i=0$, there exists a (unique) sequence $(b_i)_{0 \leq i \leq p+2}$ with $b_0=b_{p+2}=0$ and $b_i-b_{i-1}=a_i$ ($1 \leq i \leq p+2$).

With the  nodal `hat' function $\phi(x):=\left\{\begin{array}{@{}lr} x+1 & x \in [-1,0], \\ 1-x & x \in [0,1],\end{array}\right.$ and  $y^\delta$ being the continuous piecewise linear function defined by
$y^\delta(i h_\delta)= b_i$ ($0 \leq i \leq p+2$), and $y^\delta\equiv 0 $ on $((p+2)h_\delta, \mathrm{T})$, we define 
$$
w^\delta := y^\delta \otimes \phi\big(\tfrac{\bigcdot-\alpha}{H_\delta}\big) \in X^\delta.
$$
Notice that $\supp w^\delta$ is the shaded area in Figure~\ref{fig:1}. It holds that
$$
\diff_t w^\delta = h_\delta^{-1} g^\delta \otimes  \phi\big(\tfrac{\bigcdot-\alpha}{H_\delta}\big).
$$

In view of \eqref{eq:92}, we will study  $\inf_{v \in Y^\delta}  \frac{\|R_Y^{-1}\diff_t w^\delta  -v\|_Y}{\|R_Y^{-1}\diff_t w^\delta\|_Y}$.
Noticing that $R_Y^{-1}=\identity \otimes R_{H^1_0(\Omega)}^{-1}$,
from here onwards we may pretend that $\Xi=(0,H_\delta)$, and  thus $\Sigma=(0,H_\delta) \times \Omega$.

One calculates
$\|y^\delta\|_{L_2(\Xi)} \eqsim \sqrt{h_\delta} \eqsim \|g^\delta\|_{L_2(\Xi)}$ and
$|\phi\big(\tfrac{\bigcdot-\alpha}{H_\delta}\big)|_{H^1(\Omega)} \eqsim \frac{1}{\sqrt{H_\delta}}$.
With the Dirac functional $\delta_\alpha(\bigcdot):=\bigcdot(\alpha)$, we have  $\|H_\delta^{-1} \phi\big(\tfrac{\bigcdot-\alpha}{H_\delta}\big) -\delta_\alpha\|_{H^{-1}(\Omega)} \lesssim \sqrt{H_\delta}$ and, taking $\dist(\alpha,\partial\Omega) \gtrsim 1$,  $\|\boldsymbol{\delta}_{\alpha}\|_{H^{-1}(\Omega)} \eqsim 1$ , and thus $\|H_\delta^{-1}\phi\big(\tfrac{\bigcdot-\alpha}{H_\delta}\big)\|_{H^{-1}(\Omega)}\eqsim 1$.
We conclude that
\be \label{eq:91}
\|w^\delta\|_Y \eqsim \tfrac{\sqrt{h_\delta}}{\sqrt{H_\delta}},
\quad \|\diff_t w^\delta\|_{Y'} \eqsim \tfrac{H_\delta}{\sqrt{h_\delta}},
\ee
which will be used later.

 From $R_Y^{-1}$ being an isometry, for $\beta$ with $0<\beta-H_\delta-(\alpha+H_\delta) \lesssim H_\delta$  as in Figure~\ref{fig:1}, we have
\begin{align*}
\zeta_\delta&:=\frac{\|R_Y^{-1}\diff_t w^\delta-R_Y^{-1} h_\delta^{-1} g^\delta \otimes \phi\big(\tfrac{\bigcdot-\beta}{H_\delta}\big)\|_{Y}}{\|R_Y^{-1}\diff_t w^\delta\|_Y}=\frac{\|\diff_t w^\delta-h_\delta^{-1} g^\delta \otimes \phi\big(\tfrac{\bigcdot-\beta}{H_\delta}\big)\|_{Y'}}{\|\diff_t w^\delta\|_{Y'}}
\\  
& =\frac{\|H_\delta^{-1} \phi\big(\tfrac{\bigcdot-\alpha}{H_\delta}\big)-H_\delta^{-1}\phi\big(\tfrac{\bigcdot-\beta}{H_\delta}\big)\|_{H^{-1}(\Omega)}}{\|H_\delta^{-1}\phi\big(\tfrac{\bigcdot-\alpha}{H_\delta}\big)\|_{H^{-1}(\Omega)}} \lesssim \sqrt{H_\delta}
\end{align*}
by using that $|\beta-\alpha| \lesssim H_\delta$.

Next we show that $R_Y^{-1} h_\delta^{-1} g^\delta \otimes \phi\big(\tfrac{\bigcdot-\beta}{H_\delta}\big)$ is $Y$-orthogonal to $Y^\delta$.
Let $P_{(0,H_\delta)}^p$ denote the $L_2(0,H_\delta)$-orthogonal projector onto $P_p(0,H_\delta)$.
Then $P_{(0,H_\delta)}^p \otimes \identity\colon Y^\delta \rightarrow Y^\delta$, so that
$Y^\delta=\ran (P_{(0,H_\delta)}^p \otimes \identity)|_{Y^\delta} \oplus \ran (\identity-P_{(0,H_\delta)}^p \otimes \identity)|_{Y^\delta}$.
By construction of $g^\delta$, we have $\ran (P_{(0,H_\delta)}^p \otimes \identity)|_{Y^\delta} \perp_{Y}
h_\delta^{-1} g^\delta \otimes R_{H^1_0(\Omega)}^{-1} \phi\big(\tfrac{\bigcdot-\beta}{H_\delta}\big)$.
Any $v \in \ran (\identity-P_{(0,H_\delta)}^p \otimes \identity)|_{Y^\delta}$ vanishes on $(0,H_\delta) \times (\beta-H_\delta,\beta+H_\delta)$.
Consequently, for such $v$ it holds that
$$\langle h_\delta^{-1} g^\delta \otimes R_{H^1_0(\Omega)}^{-1} \phi\big(\tfrac{\bigcdot-\beta}{H_\delta}\big),v\rangle_Y=
\int_0^{H_\delta} \!\!\int_\Omega h_\delta^{-1} g^\delta(t) \phi\big(\tfrac{x-\beta}{H_\delta}\big) v(t,x)\diff t\diff x=0
$$
We conclude that indeed $Y^\delta \perp_Y R_Y^{-1} h_\delta^{-1} g^\delta \otimes \phi\big(\tfrac{\bigcdot-\beta}{H_\delta}\big)$.

From this orthogonality, the definition of $\zeta_\delta$ and an application of the triangle inequality show that
\begin{align*}
\inf_{v \in Y^\delta} \frac{\|R_Y^{-1}\diff_t w^\delta-v\|_Y}{\|R_Y^{-1}\diff_t w^\delta\|_Y}&\geq 
\inf_{v \in Y^\delta} \frac{\|R_Y^{-1} h_\delta^{-1} g^\delta \otimes \phi\big(\tfrac{\bigcdot-\beta}{H_\delta}\big)-v\|_Y}{\|R_Y^{-1}\diff_t w^\delta\|_Y} -\zeta_\delta
\\&=\frac{\|R_Y^{-1} h_\delta^{-1} g^\delta \otimes \phi\big(\tfrac{\bigcdot-\beta}{H_\delta}\big)\|_Y}{\|R_Y^{-1}\diff_t w^\delta\|_Y} -\zeta_\delta
\geq 1-2\zeta_\delta \rightarrow 1 \quad(H_\delta \rightarrow 0),
\end{align*}
or, equivalently (see \eqref{eq:92}), 
\be \label{eq:93}
\frac{\|{E_Y^\delta}'\diff_t  w^\delta\|_{{Y^\delta}'}}{\|\diff_t  w^\delta\|_{Y'}} \leq \sqrt{1-(1-2 \zeta_\delta)^2}
\rightarrow 0 \quad(H_\delta \rightarrow 0).
\ee

As a first  immediate consequence, we have
$$
\gamma_\delta^{\diff_t} \rightarrow 0 \quad (H_\delta \rightarrow 0),
$$
i.e., without the time-slab restriction, \eqref{eq:infsup2} is in general \emph{not} valid.
 
So far, we have only used that the squares of size $H_\delta$ and those of size $h_\delta$ are within spatial distance ${\mathcal O}(H_\delta)$, and that $\frac{h_\delta}{H_\delta} \leq \frac{1}{p+2}$. Under a stronger condition on $\frac{h_\delta}{H_\delta}$, as a second consequence of \eqref{eq:93} it turns out that also \eqref{eq:infsup1} is \emph{not} valid. Indeed, \eqref{eq:91} shows
$$
\frac{\|w^\delta\|_Y}{\|\diff_t  w^\delta\|_{Y'}} \rightarrow 0 \quad(\tfrac{h_\delta}{H_\delta\sqrt{H_\delta}} \rightarrow 0),
$$
and thus, because $w^\delta$ vanishes at $t=\mathrm{T}$,
$$
\frac{\|w^\delta\|_{X,\delta}^2}{\|w^\delta\|_X^2}
= \frac{\|w^\delta\|_Y^2 + \|{E_Y^\delta}' \diff_t w^\delta\|_{{Y^\delta}'}^2}{\|w^\delta\|_Y^2 + \|\diff_t w^\delta\|_{Y'}^2}
\rightarrow 0 \quad(H_\delta \rightarrow 0,\,\tfrac{h_\delta}{H_\delta\sqrt{H_\delta}} 
\rightarrow 0),
$$
or, by two applications of Theorem~\ref{thm:3},
$$
\gamma_\delta^{B_e} \rightarrow 0 \quad(H_\delta \rightarrow 0,\,\tfrac{h_\delta}{H_\delta\sqrt{H_\delta}} \rightarrow 0).
$$

\subsection{The case $d\ge2$}
Next, we discuss the extension of the above analysis to $d \geq 2$.
For $\Omega$ being a $d$-dimensional polygon, let $\cP^\delta$ be a partition of $\Sigma=(0,1) \times \Omega$ into (shape-regular, isotropic) prisms.
Let $X^\delta \subset X$ and $Y^\delta \subset Y$ be defined as before.
To substitute for $\phi\big(\tfrac{\bigcdot-\alpha}{H_\delta}\big)$ and $\phi\big(\tfrac{\bigcdot-\beta}{H_\delta}\big)$ from the $d=1$ case,
let $\phi_\alpha$ and $\phi_\beta$ be piecewise linear `hat' functions with $\int_\Omega \phi_\alpha\diff x = \int_\Omega \phi_\beta\diff x\eqsim H_\delta^d$,
$\diam \supp \phi_\alpha \eqsim H_\delta \eqsim \diam \supp \phi_\beta$,
$\dist(\supp \phi_\alpha, \supp \phi_\beta) \lesssim H_\delta$,
and such that $\cP^\delta|_{\supp \chi_{(0,H_\delta)} \otimes \phi_\beta}$ is the coarsest partition w.r.t.~which $\chi_{(0,H_\delta)} \otimes  \phi_\beta$ is piecewise polynomial, whereas $\cP^\delta|_{\supp \chi_{(0,H_\delta)} \otimes \phi_\alpha}$ is a further refinement of such a partition with a locally uniform mesh-size $h_\delta \leq \frac{H_\delta}{p+2}$, cf.~Figure~\ref{fig:1}.

Similar to the case $d=1$, we take
$$
w^\delta := y^\delta \otimes \phi_\alpha,
$$
so that $\diff_t w^\delta = h_\delta^{-1} g^\delta \otimes \phi_\alpha$. We have seen that  $\|y^\delta\|_{L_2(\Xi)} \eqsim \sqrt{h_\delta} \eqsim \|g^\delta\|_{L_2(\Xi)}$, and one calculates
$|\phi_\alpha|_{H^1(\Omega)} \eqsim H_\delta^{\frac{d}{2}-1}$, so that
\be \label{eq:65}
\|w^\delta\|_Y \eqsim \sqrt{h_\delta}\, H_\delta^{d/2-1},\quad \|\diff_t w^\delta\|_{Y'} \eqsim \tfrac{1}{\sqrt{h_\delta}}\|\phi_\alpha\|_{H^{-1}(\Omega)}
\ee
with (a lower bound for) $\|\phi_\alpha\|_{H^{-1}(\Omega)}$ to be determined.

With $\omega$ a ball of diameter $\lesssim H_\delta$ that contains $\supp \phi_\alpha \cup \supp \phi_\beta$, an application of the Poincar\'{e} estimate shows that for $v \in H^1_0(\Omega)$,
\begin{align*}
\Big|\int_\Omega (\phi_\alpha- \phi_\beta) v \diff x\Big|&=\Big|\int_\Omega \phi_\alpha\Big(v-\fint_\omega v\Big)-\phi_\beta\Big(v-\fint_\omega v\Big)\diff x\Big|
\\
&\lesssim (\|\phi_\alpha\|_{L_2(\Omega)} +\|\phi_\beta\|_{L_2(\Omega)}) H_\delta |v|_{H^1(\Omega)} \eqsim H_\delta^{d/2+1}  |v|_{H^1(\Omega)},
\end{align*}
i.e.,
$$
\| \phi_\alpha- \phi_\beta\|_{H^{-1}(\Omega)} \lesssim H_\delta^{1+d/2}.
$$

Similar to the case $d=1$, we set
\be \label{eq:58}
\zeta_\delta:=\frac{\|R_Y^{-1}\diff_t w^\delta-R_Y^{-1} h_\delta^{-1} g^\delta \otimes \phi_\beta\|_{Y}}{\|R_Y^{-1}\diff_t w^\delta\|_Y} =\frac{\|\phi_\alpha-\phi_\beta\|_{H^{-1}(\Omega)}}{\|\phi_\alpha\|_{H^{-1}(\Omega)}} \lesssim \frac{H_\delta^{1+d/2}}{\|\phi_\alpha\|_{H^{-1}(\Omega)}}.
\ee

As in the $d=1$ case, we have $Y^\delta \perp_Y R_Y^{-1} h_\delta^{-1} g^\delta \otimes \phi_\beta$, so that
$$
\frac{\|{E_Y^\delta}'\diff_t  w^\delta\|_{{Y^\delta}'}}{\|\diff_t  w^\delta\|_{Y'}} \leq \sqrt{1-(1-2 \zeta_\delta)^2}.
$$

In order to show that $\zeta_\delta \rightarrow 0$ for $H_\delta \rightarrow 0$,  it only remains to find a suitable lower bound for $\|\phi_\alpha\|_{H^{-1}(\Omega)}$. 
W.l.o.g.~ we assume that $\supp \phi_\alpha$ is contained in \mbox{$\{x \in \R^d\colon |x| \leq H_\delta\}$} and that $\{x \in \R^d\colon |x| \leq 1\} \subset \Omega$. For $d=2$, we take
$$
v(x)=v(r)=\left\{\begin{array}{cc} 1 & r \in (0,H_\delta),\\ \frac{\log r}{\log H_\delta}, &r \in (H_\delta,1),\\ 0 & r>1.\end{array}\right.
$$
Then, it holds that
\begin{align*}
\|\nabla v\|_{L_2(\Omega)^d}^2 &=\int_{r \in (H_\delta,1)} \nabla v \cdot \nabla v \diff x=\int_{r \in (H_\delta,1)} -v \Delta v \diff x+\oint_{r=H_\delta} \frac{\partial v}{\partial n} v \diff s\\
&=\oint_{r=H_\delta} \frac{\partial v}{\partial n} \diff s=
\int_0^{2 \pi} \frac{-1}{H_\delta \log H_\delta} H_\delta \diff \theta=\frac{2 \pi}{|\log H_\delta|},
\end{align*}
so that $\|\phi_\alpha\|_{H^{-1}(\Omega)} \geq \frac{\int_\Omega \phi_\alpha \diff x}{\|\nabla v\|_{L_2(\Omega)^d}}\eqsim
H_\delta^2 \sqrt{\frac{|\log H_\delta|}{2 \pi}}$,
and, using \eqref{eq:58}, indeed $\zeta_\delta \lesssim \sqrt{\frac{1}{|\log H_\delta}|} \rightarrow 0$ for $H_\delta \rightarrow 0$.

We conclude that 
$$
\gamma^{\diff_t}_\delta \rightarrow 0 \quad (H_\delta \rightarrow 0).
$$
Moreover, from \eqref{eq:65} we have $\frac{\|w^\delta\|_Y}{\|\diff_t w^\delta\|_{Y'}} \lesssim \frac{h_\delta}{H_\delta^2 \sqrt{|\log H_\delta|}}$, so that
$$
\gamma^{B_e}_\delta\rightarrow 0 \quad (H_\delta \rightarrow 0,\,\tfrac{h_\delta}{H_\delta^2 \sqrt{|\log H_\delta|}} \rightarrow 0).
$$

The analogous approach for $d \geq 3$ is to take 
$$
v(x)=v(r)=\left\{\begin{array}{cc} 1 & r \in (0,H_\delta),\\ \frac{r^{2-d}-1}{H_\delta^{2-d}-1}, &r \in (H_\delta,1),\\ 0 & r>1,\end{array}\right.
$$
which gives $\|\nabla v\|_{L_2(\Omega)^d}^2 \eqsim H_\delta^{d-2}$, so that $\|\phi_\alpha\|_{H^{-1}(\Omega)} \geq \frac{\int_\Omega \phi_\alpha \diff x}{\|\nabla v\|_{L_2(\Omega)^d}}\eqsim H_\delta^{1+d/2}$, which together with \eqref{eq:58} unfortunately only gives $\zeta_\delta \lesssim 1$ for $H_\delta \rightarrow 0$.

\end{document}